\documentclass[11pt]{article}
\usepackage{listings}
\usepackage{pdflscape}
\usepackage{epsfig}
\usepackage{lscape}
\usepackage{longtable}
\usepackage{amsmath,amssymb,amsthm}
\usepackage{graphicx}
\usepackage{subfigure}
\usepackage{color}
\usepackage{placeins}
\usepackage{url}
\usepackage{cases}

\newcommand{\blue}[1]{\begin{color}{blue}#1\end{color}}

\newcommand{\red}[1]{\begin{color}{red}#1\end{color}}
\newcommand{\green}[1]{\begin{color}{green}#1\end{color}}

\def\cA{{\cal A}}
\def\cB{{\cal B}}
\def\cL{{\cal L}}
\def\cI{{\cal I}}
\def\cS{{\cal S}}
\def\cP{{\cal P}}
\def\cD{{\cal D}}

\def\cO{{\cal O}}
\def\cQ{{\cal Q}}
\def\cT{{\cal T}}
\def\cM{{\cal M}}
\def\cU{{\cal U}}
\def\cV{{\cal V}}

\def\norm#1{\|#1 \|}
\def\inprod#1#2{\langle#1,\,#2 \rangle}
\def\vec{{\rm {\bf vec}}}
\def\Range{\textup{Ran}}

\newtheorem{theorem}{Theorem}

\newtheorem{lemma}{Lemma}
\newtheorem{remark}{Remark}

\newtheorem{proposition}{Proposition}

\def\mc{\multicolumn}

\begin{document}

\title{\bf On the efficient computation of a generalized Jacobian  of the projector over the Birkhoff polytope}
%\author{Xudong Li, Defeng Sun and Kim-Chuan Toh\\ National University of Singapore}
\author{Xudong Li\thanks{Department of Operations Research and Financial Engineering, Princeton University, Sherrerd Hall 116, Princeton, NJ 08544 ({\tt xudongl@princeton.edu}).}, \;
Defeng Sun\thanks{Department of Applied Mathematics, The Hong Kong Polytechnic University, Hung Hom,
 Hong Kong ({\tt defeng.sun@polyu.edu.hk}. On leave from Department  of  Mathematics,   National University of Singapore).} \  and \
Kim-Chuan Toh\thanks{Department of Mathematics and Institute of Operations Research and Analytics, National University of Singapore, 10 Lower Kent Ridge Road, Singapore
({\tt mattohkc@nus.edu.sg}).  }
}

\date{April 18, 2018}
\maketitle
\begin{abstract}
  We derive an explicit formula, as well as  an efficient procedure, for constructing
  a generalized  Jacobian for the projector of a given square matrix  onto the Birkhoff polytope, i.e.,
  the set of doubly stochastic matrices. To guarantee the high efficiency of our procedure, a   semismooth Newton method for
  solving the dual of the projection problem is proposed and efficiently implemented. Extensive numerical experiments
  are presented to demonstrate the merits and effectiveness of our method by comparing its performance against other powerful solvers such as the commercial software Gurobi and the academic code
  PPROJ [{\sc Hager and Zhang}, SIAM Journal on Optimization, 26 (2016), pp.~1773--1798].
  In particular, our algorithm is able to solve the projection problem with
  over one billion variables and nonnegative constraints to a very high accuracy in less than 15 minutes on a modest desktop computer.
  {More importantly,  based on our  efficient computation of the projections and their generalized Jacobians, we can
   design} a highly efficient augmented Lagrangian method (ALM) for
  solving a class of convex quadratic programming (QP) problems constrained by the Birkhoff polytope.
 The resulted  ALM   is demonstrated to be much more efficient than Gurobi
  in solving a collection of
  QP problems arising from the relaxation of quadratic assignment problems.
\end{abstract}
\noindent
\textbf{Keywords:}
Doubly stochastic matrix, semismoothness, Newton's method, generalized Jacobian
\textbf{AMS subject classifications:} 90C06, 90C20, 90C25, 65F10.

%%%%%%%%%%%%%%%%%%%%%%%%%%%%%%%%%%
\section{Introduction}

The Birkhoff polytope is  the set of $n\times n$ doubly stochastic  matrices {defined by}
\[\mathfrak{B}_n := \{X\in\Re^{n\times n}\mid Xe = e,\, X^T e = e, X\ge 0\},\]
where $e\in\Re^n$ is the vector of all ones {and $X \ge 0$ means that all the elements of $X$ are
nonnegative}.
In this paper, we
focus on the  problem of projecting a matrix $G\in \Re^{n\times n}$ onto the  Birkhoff polytope $\mathfrak{B}_n$, i.e., solving  the following special  convex quadratic programming (QP) problem
\begin{equation} \label{eq:Proj_Birkhoof}
   \min \Big\{\frac{1}{2}\norm{X- G}^2
 \mid  X \in \mathfrak{B}_n \Big\},
\end{equation}
%where  $G\in\Re^{n\times n}$ is a given matrix and
where $\norm{\cdot}$ denotes the Frobenius norm.
The optimal solution of \eqref{eq:Proj_Birkhoof}, i.e., the Euclidean projection of $G$ onto  $\mathfrak{B}_n$,
is denoted by $\Pi_{\mathfrak{B}_n}(G)$.

The Birkhoff polytope has long been an important object in statistics, combinatorics, {physics} and optimization. As the convex hull of the set of permutation matrices \cite{Birkhoof1946three,Neumann1953certain}, the Birkhoff polytope has frequently been used to derive relaxations of nonconvex  optimization problems involving permutations, such as the quadratic assignment problems {\cite{jiang2016lpnorm}} and the
seriation problems {\cite{fogel2013convex,lim2014beyond}}. Very often the algorithms that are designed to solve these relaxed problems
need to compute the projection of matrices onto the polytope $\mathfrak{B}_n$ {\cite{fogel2013convex,jiang2016lpnorm}}.
On the other hand, the availability of a fast solver for computing
$\Pi_{\mathfrak{B}_n}(\cdot)$ can also influence how one would design an algorithm
to solve the relaxed problems.
As we shall demonstrate later, indeed one can design a highly
efficient algorithm to solve QP problems involving Birkhoff polytope constraints
if a fast solver for computing $\Pi_{\mathfrak{B}_n}(\cdot)$ {and its generalized Jacobian} is readily available.

%usually need to maintain the produced iterates in the Birkhoff polytope. In fact, problem \eqref{eq:Proj_Birkhoof} are usually appears to be the subproblems involved in these algorithms.
%These problems include linear and quadratic assignment problems, seriation problems and graph matching problems.

Let $D$ be a  nonempty  polyhedral convex set. Besides the computation of the {Euclidean} projector $\Pi_D(\cdot)$,  the  differential properties of the projector have long been recognized to be  important  {in nonsmooth analysis and algorithmic design}.
In \cite{Haraux1977how}, Haraux showed that the projector onto a polyhedral convex set must be directionally differentiable. Pang \cite{Pang1990Newton}, inspired by an unpublished report of Robinson \cite{Robinson1985implicit}, derived an explicit formula for the directional {derivative} and discussed
the {Fr{\'e}ch{e}t} differentiability of the projector.
By using the piecewise linear structure of   $\Pi_D(\cdot)$, one may further use the results of Pang and Ralph
\cite{PRalph1996}  to characterize the
B-subdifferential and the corresponding Clarke generalized Jacobian \cite{clarke1990optimization} of the projector.
However, for an arbitrary polyhedral set $D$, the calculations of these
generalized Jacobians are generally very difficult to accomplish numerically, if feasible at all.
In order to circumvent this difficulty, Han and Sun in \cite{han1997newton} proposed a special multi-valued mapping as a {more} tractable replacement for the generalized Jacobian and used it in the design of the generalized Newton and quasi-Newton methods for solving a class of piecewise  smooth equations.
 The idea of    getting  an element from {the} aforementioned multi-valued mapping in \cite{han1997newton}
 is to find certain dual multipliers of the projection problem together with a corresponding set of linearly independent active constraints.
Since the linear independence checking   can be costly, in particular when  the dimension of the underlying  projection problem is large,
in this paper, we {aim at introducing} a technique to  avoid  this checking and provide an efficient computation of a generalized Jacobian in the sense of \cite{han1997newton} for the Euclidean projector over the polyhedral convex set with an emphasis on the Birkhoff polytope.
 We achieve this goal by deriving an explicit formula for constructing a special generalized Jacobian in the sense of \cite{han1997newton}.
In addition, based on the special structure of the Birkhoff polytope, we further simplify the formula and discuss efficient implementations for its calculation.
We shall emphasize here that, in contrast to the previous work done in  \cite{han1997newton} and as a surprising result, our specially constructed Jacobian needs neither the knowledge of the
dual multipliers associated with the projection problem nor the set of corresponding linearly independent active constraints.
%In fact, our
%This is .
%  \fbox{Our formula for the special generalized Jabcobian thus come to us as a surprise.}
%  \fbox{May need to rewrite.}

As one can see later, the computation of the {Euclidean} projector $\Pi_{D}(\cdot)$ is one of
 the key steps in our construction of the aforementioned special generalized Jacobian.
 Hence, its %effectiveness
efficiency
  is crucial to our construction. As a simple yet fundamental convex quadratic programming problem, various well developed algorithms have been used for computing the projection onto a polyhedral convex set such as the state-of-the-art interior-point based commercial solvers Gurobi \cite{Gurobi2016} and CPLEX\footnote{\url{https://www-01.ibm.com/software/commerce/optimization/cplex-optimizer/index.html}}. Recently, Hager and Zhang \cite{zhang2016projection} proposed to compute the projector through the dual approach by combining the sparse reconstruction by separable approximation (SpaRSA) \cite{wright2009sparse} and the dual active set algorithm. An efficient implementation called PPROJ is also provided in \cite{zhang2016projection} and the comparisons between PPROJ and CPLEX indicate that PPROJ is robust,
accurate and fast. In fact, the dual approach for solving Euclidean projection problems has been extensively studied in the literature. For example, both the dual quasi-Newton method \cite{Malic2004dual} and the dual semismooth Newton method \cite{qi2006quadratically} have been developed to compute the Euclidean projector onto the intersection of an affine subspace and a closed convex cone.
{Another popular method for computing
the projection over the intersection of an affine subspace and a closed convex cone is the
{alternating projections method with Dykstra's correction \cite{dykstra1983algorithm}}
%\red{Dykstra's correction \cite{dykstra1983algorithm} based alternating projections method}
that was
 proposed in \cite{higham2002computing}. It has been shown in \cite[Theorem 5.1]{Malic2004dual} that the {alternating projections method with Dykstra's correction \cite{dykstra1983algorithm}} is a dual gradient method  with constant step size. As can be observed from the numerical comparison in \cite{qi2006quadratically}, the semismooth Newton method outperformed the
 quasi-Newton and Dykstra's methods by a significant margin.}

%\red{Due to its great potential in different applications and mathematical elegance, our preliminary goal in this paper is to derive a derive a highly efficient procedure to compute the special generalized Jacobian of $\Pi_{\mathfrak{B}_n}$ by leveraging on its particular structure.}
%\red{In this paper, due to its great potential in different applications and mathematical elegance, we focus on the
%case where the polyhedral convex set $D$ is chosen to be $\mathfrak{B}_n$. Then, we are able to derive a highly efficient procedure to compute the special generalized Jacobian of $\Pi_{\mathfrak{B}_n}$ by leveraging on the particular structure of the Birkhoff polytope.}

{As already mentioned in the second paragraph above, the projection onto the Birkhoff polytope
has important applications in different areas. It is also by itself a mathematically elegant problem to study. Thus
in this paper,  we shall focus on the
case where the polyhedral convex set $D$ is chosen to be the Birkhoff polytope $\mathfrak{B}_n$.
Due to the elegant structure of $\mathfrak{B}_n$,
we are able to derive a highly efficient procedure to compute a special generalized Jacobian of $\Pi_{\mathfrak{B}_n}$ by leveraging on its structure.}
As a crucial step in our procedure, we
choose to use the semismooth Newton method for computing the projector $\Pi_{\mathfrak{B}_n}(\cdot)$ via solving the dual of the projection problem \eqref{eq:Proj_Birkhoof} and provide a highly efficient implementation. Extensive numerical experiments are presented to demonstrate the merits and effectiveness of our method by comparing its performance against other solvers such as Gurobi and  PPROJ.
 In particular, our algorithm is able to solve a projection problem over the Birkhoff polytope with
  over one billion variables and nonnegative constraints to a very high accuracy  in less than 15 minutes
  on a modest desktop computer.
  In order to further demonstrate the importance of our procedure, we  also propose a highly efficient augmented Lagrangian method (ALM) for  solving a class of convex QP problems with Birkhoff polytope constraints.
  Our ALM  is demonstrated to be much more efficient than Gurobi in solving a collection of QP problems arising from the relaxation of quadratic assignment problems.

The remaining parts of this paper are organized as follows. The next section is devoted to studying the generalized
Jacobians of the projector onto a general polyhedral convex set. In Section \ref{sec:projBP}, a semismooth Newton method is designed for projecting a matrix onto the Birkhoff {polytope}. Then, a generalized Jacobian of the projector at the given matrix is computed. Efficient implementations of these steps are discussed.  In Section \ref{sec:QP_BP}, we show how the generalized Jacobian obtained in Section \ref{sec:projBP} can be employed in the design of
a highly efficient augmented Lagrangian method for solving convex quadratic programming problems with Birkhoff polytope constraints. In Section \ref{sec:numer},
we conduct numerical experiments to evaluate the performance of our algorithms against Gurobi and PROJ
for computing the projection onto the Birkhoff polytope.
The advantage of incorporating the second order (generalized Hessian) information
into the design of an algorithm for solving
 convex quadratic programming problems with Birkhoff polytope constraints
is also demonstrated.
We conclude our paper in the final section.

Before we move to the next section, here we list some {notation} to be used later in this paper. For any given $X\in\Re^{m\times n}$, we define its associated  vector ${\vec(X)}\in\Re^{mn}$  by
\[\vec(X):=[x_{11},\ldots,x_{m1}, x_{12},\ldots, x_{m2},\ldots.
x_{1n},\ldots,x_{mn}].\]
For any given vector $y\in\Re^{n}$, we denote by ${\rm Diag}(y)$ the diagonal matrix whose $i$th diagonal element is given by $y_i$. For any given matrix $A\in\Re^{m\times n}$, we use $\Range(A)$ and ${\rm Null}(A)$ to denote the range space and the null space of $A$, {respectively}. Similar {notation} is used when $A$ is replaced by a linear operator
$\cA$. We use $I_n$ to denote the $n$ by $n$ identity matrix in $\Re^{n\times n}$ and  $N^{\dagger}$ to denote the Moore-Penrose pseudo-inverse of a given matrix $N\in\Re^{m\times n}$.

%%%%%%%%%%%%%%%%%%%%%%%%%%%%%%%%%%%%%%%%%%
\section{Generalized Jacobians of the projector over polyhedral convex sets}

In this section, we study the variational properties of the projection mapping $\Pi_{D}(\cdot)$ for a  nonempty polyhedral convex set $D \subseteq \Re^{n}$ expressed in the following form
\begin{equation}
\label{setD}
D:=\{x\in\Re^n \mid A x \ge b, \, Bx = d\},
\end{equation}
where $A\in\Re^{m\times n}, B\in\Re^{p\times n}$ and $b\in\Re^m, d\in\Re^{p}$ are given data. Without  loss of generality, we assume that ${\rm rank}(B) = p$, $p\le n$.

%Note that the Birkhoff polytope is also in the form of $D$.
Given $x\in\Re^n$, from the representation of $D$ in \eqref{setD}, we know that there exist multipliers $\lambda\in \Re^{m}$ and $\mu \in \Re^p$ such that
\begin{equation}\label{eq:KKTprojD}
\left\{
\begin{aligned}
&\Pi_{D}(x) - x + A^T\lambda + B^T\mu = 0,\\[5pt]
&A\Pi_{D}(x) - b \ge 0, \quad B\Pi_{D}(x) - d = 0, \\[5pt]
& \lambda \le 0, \quad \lambda^T(A\Pi_{D}(x) - b) = 0.
\end{aligned}
\right.
\end{equation}
Let $M(x)$ be the set of multipliers associated with $x$, i.e.,
\[M(x):=\{(\lambda,\mu)\in\Re^m\times\Re^p\mid (x,\lambda, \mu)\; {\rm satisfies}\; \eqref{eq:KKTprojD}\}.\]
Since $M(x)$ is a nonempty polyhedral convex set containing no lines, it has at least one extreme point denoted   as $(\bar \lambda, \bar \mu)$ \cite[Corollary 18.5.3]{rockafellar1970convex}.
Denote
\begin{equation}\label{eq:Ix}
I(x): = {\{i \in \{ 1,\ldots,m \} \mid A_i\Pi_D(x) = b_i\}},
\end{equation}
where $A_i$ is the $i$th row of the matrix $A$.
Define a collection of index sets:
\begin{equation*}
\begin{aligned}
\cD(x): =  \{\; & K\subseteq \{1,\ldots,m\} \mid \exists\, (\lambda,\mu)\in M(x) \; {\rm s.t.}\; {\rm supp}(\lambda) \subseteq K \subseteq I(x),\\[5pt]
\;&  [A_K^T\ B^T]\;\mbox{\rm is of  full column rank}
\},
\end{aligned}
\end{equation*}
where ${\rm supp}(\lambda)$ denotes the support of  $\lambda$, i.e., the set of indices $i$ such that $\lambda_i \neq 0$ and $A_K$ is the matrix consisting of the rows of $A$, indexed by $K$. As is already noted in \cite{han1997newton}, the set $\cD(x)$ is nonempty due to the existence of the extreme point
$(\bar \lambda, \bar \mu)$ of $M(x)$. Since it is usually difficult to
 calculate the B-subdifferential $\partial_B \Pi_{D}(x)$ or the Clarke generalized Jacobian $\partial \Pi_{D}(x)$ for a general polyhedral convex set $D$ at a given point $x$,
 Han and Sun in \cite{han1997newton} introduced the following multi-valued mapping
 $\cP:\Re^n\rightrightarrows\Re^{n\times n}$ as a computable  replacement for $\partial_B\Pi_{D}(\cdot)$,
 namely,
\begin{equation}
\label{eq:genJacobianPx}
\cP(x) := \left\{
P\in\Re^{n\times n}\mid P = I_n - [A_K^T\ B^T]
\left( \left[   \begin{array}{c}
A_K \\
B \\
\end{array}
\right]
[A_K^T\ B^T]\right)^{-1}\left[\begin{array}{c}
A_K \\
B \\
\end{array}
\right],\, K\in\cD(x)
\right\}.
\end{equation}
The mapping $\cP$ has a few important properties \cite{han1997newton}{, which} are summarized in the following proposition.
%We summarize nice properties of $\cP(\cdot)$ presented in \cite{han1997newton} in the following theorem.
\begin{proposition}\label{thm:HanSun1997}
	For any $x\in \Re^n$, there exists a neighborhood $U$ of $x$ such that
	\[\cD(y)\subseteq \cD(x),\quad \cP(y)\subseteq \cP(x), \quad \forall y\in U.\]
	If $\cD(y)\subseteq \cD(x)$, it holds that
	\[\Pi_D(y) = \Pi_{D}(x) + P(y-x),\quad {\forall P\in\cP(y)}.\]
	Thus, $\partial_B \Pi_{D}(x) \subseteq \cP(x)$.
\end{proposition}
Note that even {with formula \eqref{eq:genJacobianPx}}, for a given point $x\in \Re^n$, it is still not easy to find an
element in $\cP(x)$ as one needs to find a suitable index $K \in \cD(x)$ corresponding
to some multiplier $(\lambda, \mu)\in M(x)$.
%Next, we show that in fact we can construct
A key contribution  made in this paper is that we are able to construct
a matrix $P_0\in\Re^{n\times n}$ such that $P_0\in \cP(x)$ without knowing the index set $K$ and its corresponding multipliers.
In addition, we show how to efficiently compute the matrix
$P_0$ when the polyhedral set $D$ possesses certain special structures. We shall emphasize here that the efficient computation of $P_0$ is crucial in the design of various second order algorithms
for solving optimization problems involving the polyhedral constraint $x\in D$.

We  present here a very useful lemma which will be used extensively in our later discussions.
\begin{lemma}\label{lemma:MPInveqProj}
	Let $H\in\Re^{m\times n}$ be a given  matrix and $\widehat H \in \Re^{m_1\times n}$ be a full row rank matrix satisfying ${\rm Null}(\widehat H)= {\rm Null}(H)$. Then  it holds that
	\[H^T(H H^T)^{\dagger}H = \widehat{H}^T(\widehat{H}\widehat{H}^T)^{-1}
	\widehat{H}.\]
\end{lemma}
\begin{proof}
	By the singular value decomposition   of $H$ and the definition of the Moore-Penrose pseudo-inverse of $HH^T$, we can obtain through some simple calculations that
	\begin{equation}\label{eq:MPinvd}
	H^T(H H^T)^{\dagger}H\, d = \Pi_{\Range(H^T)}(d),\quad  \forall d\in\Re^n.
	\end{equation}
	Meanwhile, since $\widehat H$ is of  full row rank, we know (e.g., see  \cite[Page 46 (6.13)]{trefethen1997numerical}) that
	\begin{equation}\label{eq:invd}
	\widehat{H}^T(\widehat{H}\widehat{H}^T)^{-1}
	\widehat{H} \, d = \Pi_{\Range(\widehat H^T)}(d),\quad  \forall d\in\Re^n.
	\end{equation}
	Equations \eqref{eq:MPinvd} and \eqref{eq:invd}, together with the fact that
	\[
	{\Range}(H^T) = {\rm Null}(H)^\perp = {\rm Null}(\widehat H)^\perp = {\Range}(\widehat H^T),
	\]
	imply the desired result.
\end{proof}

\begin{theorem}
	\label{thm:elementinJocobianPx}
	For any given $x\in\Re^n$, let $I(x)$ be given in \eqref{eq:Ix}. Denote
	\begin{equation}\label{eq:P0Jacobian}
	P_{0} := I_n - [A_{I(x)}^T\ B^T]
	\left( \left[   \begin{array}{c}
	A_{I(x)} \\
	B \\
	\end{array}
	\right]
	[A_{I(x)}^T\ B^T]\right)^{\dagger}\left[\begin{array}{c}
	A_{I(x)} \\
	B \\
	\end{array}
	\right].
	\end{equation}
	Then, $P_0 \in \cP(x)$.
\end{theorem}
\begin{proof}
	Let $(\bar \lambda, \bar \mu)$ be an extreme point of $M(x)$. Denote $\overline{K}:= {\rm supp}(\bar \lambda)$. Then, $\overline K \subseteq I(x)$. From the definition of extreme points, we observe that
	$[A_{\overline K}^T\ B^T]$ has linearly independent columns.
%	\blue{Hence $\overline{K}\in \cD(x)$.}
By adding indexes from $I(x)$  to $\overline K$ if necessary, one can obtain an index
	set $K$ such that $\overline K \subseteq K \subseteq I(x)$, $[A_K^T\ B^T]$ has full column rank and
	\begin{equation}
	\label{eq:RangeAKeqRangeAI}
	{\Range}([A_K^T\ B^T]) = {\Range}([A_{I(x)}^T\ B^T]).
	\end{equation}
	That is, $K\in \cD(x)$. Therefore,
	\begin{equation*}
	P := I_n - [A_{K}^T\ B^T]
	\left( \left[   \begin{array}{c}
	A_{K} \\
	B \\
	\end{array}
	\right]
	[A_{K}^T\ B^T]\right)^{-1}\left[\begin{array}{c}
	A_{K} \\
	B \\
	\end{array}
	\right] \in \cP(x).
	\end{equation*}
	By \eqref{eq:RangeAKeqRangeAI} and Lemma \ref{lemma:MPInveqProj}, we know that
	\[P_0 = P.\] Thus, $P_0\in \cP(x)$ and this completes the proof for the theorem.
\end{proof}
\begin{remark}\label{rmk:singular_B}
	From Lemma \ref{lemma:MPInveqProj}, we know that the matrix $B$ in \eqref{eq:P0Jacobian} can be replaced by any matrix $\widehat B$ satisfying
	${\rm Null}(B) = {\rm Null}(\widehat B)$. In fact, Lemma \ref{lemma:MPInveqProj} and Theorem
	 \ref{thm:elementinJocobianPx} together imply that $P_0$ is invariant with respect to the algebraic
	 representation of the polyhedral convex set $D$,
	 i.e., it is in fact a geometric quantity corresponding to $D$ at $x$.
\end{remark}
%\begin{lemma}
%	Let $A\in\Re^{m\times n}$ be a given matrix. For any matrix $\hat A\in \Re^{\hat m \times n}$ satisfying ${\rm Null}(A) = {\rm Null}(\hat A) $, it holds that
%	\[ {\rm Null}(AB) = {\rm Null}(\hat A B) , \quad \forall B\in \Re^{n\times l}.\]
%\end{lemma}
{In general, it is not clear whether $P_0$ is a {Clarke} generalized Jacobian. As a computable replacement,
the matrix $P_0$ will be referred to as the HS-Jacobian  of  $\Pi_D$ at $x$ in the sense of \cite{han1997newton}}.
Apart from the calculation of $\Pi_{D}(x)$, one can observe from Theorem \ref{thm:elementinJocobianPx} that the
key step involved in the computation of $P_0$ is the computation of the Moore-Penrose pseudo-inverse in \eqref{eq:P0Jacobian}.
Next, we show that when the matrix $A$ in the inequality constraints in \eqref{setD} is the identity matrix, the
procedure for computing $P_0$ can be further simplified.

\begin{proposition}\label{prop:equi_A}
	Let $\theta\in \Re^n$ be a given vector with each entry $\theta_i$ being $0$ or $1$  for all $i=1,\ldots,n$. Let $\Theta = {\rm Diag}(\theta)$ and $\Sigma = I_n - \Theta$. For any given matrix $H\in\Re^{m\times n}$, it holds that
	\begin{equation}
	\label{eq:genH_S2L}
	P :=  I_{n} - [\Theta\ H^T] \left( \left[
	\begin{array}{c}
	\Theta \\
	H \\
	\end{array}
	\right]
	[\Theta\ H^T]\right)^{\dagger} \left[
	\begin{array}{c}
	\Theta \\
	H \\
	\end{array}
	\right] = \Sigma - \Sigma H^T(H\Sigma H^T)^{\dagger} H\Sigma.
	\end{equation}
\end{proposition}
\begin{proof} We only consider {the case when} $\Theta \ne 0$ as the conclusion holds trivially if $\Theta =0$.
	%Let us denote by $P$ the left-hand side matrix of equation \eqref{eq:genH_S2L}.
	From Lemma \ref{lemma:MPInveqProj}, we observe that
	$P$ is the orthorgonal projection onto ${\rm Null} \left[\begin{array}{c}
	 \Theta \\
	H \\
	\end{array}
	\right]$. Hence \[{\rm Ran}(P) {=} {\rm Null} \left[\begin{array}{c}
	 \Theta \\
	H \\
	\end{array}
	\right] \subset {\rm Null}(\Theta) = {\rm Ran}(\Sigma),\] where
	the last equality comes from the definitions of $\Theta$ and $\Sigma$.	
	From here, it is easy to show that
	 $P = \Sigma P$.
	  Since $P$ is a symmetric matrix, it further holds that
	\begin{equation}\label{eq:peqsigpsig}
	P = \Sigma P \Sigma.
	\end{equation}
{Let $\widehat{\Theta}$ be the submatrix of $\Theta$ formed by deleting all the zero rows of $\Theta$. Then, it is readily shown that
	\[ {\rm Null}(\Theta) = {\rm Null}(	\widehat \Theta), \quad 	\widehat \Theta 	\widehat \Theta^T = I_r, \quad 	\widehat \Theta^T 	\widehat \Theta = \Theta,\]
	where $r$ is the number of rows of $\widehat \Theta$.}
	%Denote by $\widehat{\Theta}$ a submatrix of $\Theta$ constructed
	Let $\left[\begin{array}{c}
	\widehat \Theta \\
		\widehat H \\
	\end{array}
	\right]$ be a full row rank matrix such that
	%${\rm Null}(\Theta) = {\rm Null}(	\widehat \Theta)$, and
	${\rm Null} \left[\begin{array}{c}
		\widehat \Theta \\
		\widehat H \\
	\end{array}
	\right]
	= {\rm Null} \left[\begin{array}{c}
	\Theta \\
	H \\
	\end{array}
	\right]. $
	Then, by Lemma \ref{lemma:MPInveqProj} and \eqref{eq:peqsigpsig}, we know that
	\begin{equation*}
	P =I_{n} - [	\widehat \Theta^T\ 	\widehat H^T] M^{-1} \left[
	\begin{array}{c}
		\widehat \Theta \\
		\widehat H  \\
	\end{array}
	\right]
	%	By the projection Lemma, we know that ${\rm Range }(P) \subseteq {\rm Range}(\Sigma)$. Therefore, by the definition of $\Sigma$, we have that $P = \Sigma P$. That is,
	= \Sigma - [0\ \Sigma 	\widehat H^T]  M^{-1} \left[
	\begin{array}{c}
	0 \\
		\widehat H \Sigma \\
	\end{array}
	\right],
	\end{equation*}
	where
	\[M: = \left[
	\begin{array}{c}
		\widehat \Theta \\
		\widehat H \\
	\end{array}
	\right][	\widehat \Theta^T\ 	\widehat H^T] =
	\left[
	\begin{array}{cc}
	{I_r} & 	\widehat \Theta 	\widehat H^T \\
		\widehat H 	\widehat \Theta^T & 	\widehat H	\widehat H^T \\
	\end{array}
	 \right].
	\]
	Therefore, we only need to focus on the $(2,2)$ block of the
	 inverse of the partitioned matrix $M$. Simple calculations show that
	\[  (M^{-1})_{22} = (	\widehat H	\widehat H^T - 	\widehat H	\widehat \Theta^T 	\widehat \Theta 	\widehat H^T)^{-1} = (	\widehat H\Sigma 	\widehat H^T)^{-1}.\]
	Therefore,
	\[P = \Sigma - \Sigma 	\widehat H^T(	\widehat H \Sigma 	\widehat H^T)^{-1}	\widehat H \Sigma.\] The desired result then follows directly from Lemma \ref{lemma:MPInveqProj} since
	${\rm Null}(	\widehat H\Sigma) = {\rm Null}(H\Sigma)$.
\end{proof}

{We should emphasize that the above proposition is particularly useful for
calculating the HS-Jacobian of the projection over a polyhedral set defined by
the intersection of hyperplanes and the nonnegative orthant.
In particular, we will see how the proposition is applied to
compute the HS-Jacobian of $\Pi_{\mathfrak{B}_n}$
in the next section. Here we provide a  proposition  on the projection over the general polyhedral set rather   than on the Birkhoff polytope only as we believe that it can be useful in other situations.}

%%%%%%%%%%%%%%%%%%%%%%%%%%%%%%%%%%
\section{Efficient procedures for computing $\Pi_{\mathfrak{B}_n}(\cdot)$ and its HS-Jacobian}
\label{sec:projBP}

In this section, we focus on the projection over the Birkhoff polytope $\mathfrak{B}_n$ and calculate the associated HS-Jacobian
by employing the efficient procedure developed in Theorem \ref{thm:elementinJocobianPx} and Proposition
\ref{prop:equi_A}. As a by-product, we also describe and
implement a highly efficient algorithm for computing the projection $\Pi_{\mathfrak{B}_n}(G) $, i.e., the optimal solution for  problem \eqref{eq:Proj_Birkhoof} with a given matrix $G\in\Re^{n\times n}$.
%In this section,
%\magenta{by employing the efficient computation of the HS-Jacobian
%in Theorem \ref{thm:elementinJocobianPx} and Proposition
%\ref{prop:equi_A},
%we shall design an efficient algorithm for
%computing the projection $\Pi_{\mathfrak{B}_n}(G) $
%for the problem \eqref{eq:Proj_Birkhoof} and study the variation properties of $\Pi_{\mathfrak{B}_n}(\cdot)$.}

Let the linear operator $\cB:\Re^{n\times n}\to \Re^{2n}$ be defined by
\[\cB(X) := [e^T X^T\  e^TX ]^T,\quad X\in\Re^{n\times n}.\]
Then, problem \eqref{eq:Proj_Birkhoof} can be represented as
\begin{equation}\label{eq:newton_proj_P}
\min\, \left\{ \frac{1}{2} \norm{X - G}^2 \,\mid\, \cB X = b, \, X\in C  \right\},
\end{equation}
where $b := [e^T \ e^T]^T\in \Re^{2n}$ and $C:= \{X\in\Re^{n\times n}\mid X\ge 0\}$.
Note that $b\in {\rm Ran}(\cB)$ and ${\rm dim}({\rm Ran}(\cB)) = 2n-1$.
%where $\Range(\cB)$ denotes the range space of $\cB$.

{
Suppose that $\overline G := \Pi_{\mathfrak{B}_n}(G)$ has been computed.  {We then aim} to find the HS-Jacobian of  $\Pi_{\mathfrak{B}_n}$ at the given point $G$.
%Denote $\overline G := \Pi_{\mathfrak{B}_n}(G)$.
Define  the linear operator $\Xi:\Re^{n\times n} \to \Re^{n\times n}$ by
\begin{equation}\label{eq:cAH}
\Xi(H) := H - \Theta^G \circ H, \quad \, H\in\Re^{n\times n},
\end{equation}
where $\Theta \in \Re^{n\times n}$ is given {as follows. For} all $1\le i,j\le n$,
\begin{equation*}
\Theta_{ij}^G =  \left\{
\begin{aligned}
&1, \quad {\rm if} \; \overline G_{ij} = 0, \\[5pt]
&0, \quad {\rm otherwise.}
\end{aligned}
\right.
\end{equation*}

\begin{proposition}\label{prop:HS_Jacobian_DSP}
	Given $G\in\Re^{n\times n}$, let $\Xi$ be the linear operator defined in
	\eqref{eq:cAH}.
	Then the  linear operator $\cP:\Re^{n\times n}\to \Re^{n\times n}$ given by
	\begin{equation}\label{eq:HS-Jocobian_cPH}
	\cP(H):= \Xi(H) - \Xi\cB^*(\cB\Xi\cB^*)^{\dagger}\cB\Xi(H),\quad \forall\, H\in\Re^{n\times n},
	\end{equation}
	is the HS-Jacobian of  $\Pi_{\mathfrak{B}_n}$ at $G$. Moreover, $\cP$ is self-adjoint and positive semidefinite.
\end{proposition}
\begin{proof}
	The desired result follows directly from Theorem \ref{thm:elementinJocobianPx},  Proposition \ref{prop:equi_A} and
	Remark \ref{rmk:singular_B}.
\end{proof}
}

{Next, we focus on designing an efficient algorithm for computing the optimal solution of problem \eqref{eq:newton_proj_P}, i.e, the projection $\Pi_{\cB_n}(G)$.}
By some simple calculations, we can derive {a} dual of \eqref{eq:newton_proj_P} in the minimization form as follows:
\begin{equation}
\label{eq:newton_proj_D}
\min\, \left\{\varphi(y):= \frac{1}{2} \norm{\Pi_{C}(\cB^*y + G)}^2  - \inprod{b}{y} - \frac{1}{2}\norm{G}^2 \,\mid\, y\in \Range(\cB)\right\}.
\end{equation}
With no difficulty, we can write down the KKT conditions associated with problems \eqref{eq:newton_proj_P} and \eqref{eq:newton_proj_D} as follows:
\begin{equation}
\label{eq:KKT-PD}
X = \Pi_{C}(\cB^*y + G),\quad \cB X = b, \quad y\in\Range(\cB).
\end{equation}
{Note that the subspace constraint $y\in \Range(\cB)$ is imposed to ensure the boundedness of the solution set of \eqref{eq:newton_proj_D}.} Indeed, since ${\rm int}(C) \neq \emptyset$ and $\cB:\Re^{n\times n}\to \Range(\cB)$ is surjective, we have from
	\cite[Theorem 2.165]{bonnans2013perturbation} that the solution set to the KKT system \eqref{eq:KKT-PD} is nonempty and for any $\tau \in \Re$ the level set $\{y\in\Range(\cB)\mid \varphi(y)\le \tau \}$ is convex, closed and bounded.

%In this section, we make the following standing assumptions regarding problems \eqref{eq:newton_proj_P}.
%\begin{assumption}
%  \label{assump:Slater}
%  Problem \eqref{eq:newton_proj_P} satisfies the Slater condition:
%  \begin{equation*}
%    \exists \, X_0\in {\rm int}(C) \quad {\rm such \; that} \quad \cB X_0 = b.
%  \end{equation*}
%\end{assumption}

Note that $\varphi(\cdot)$ is convex and continuously differentiable on $\Range(\cB)$ with
\[\nabla \varphi(y) = \cB \Pi_{C}(\cB^* y + G) - b, \quad \forall\, y\in\Range(\cB).\]
Let $\bar y$ be a  solution to the following nonsmooth equation
\begin{equation}\label{eq:nablaphi}
\nabla \varphi(y) = 0, \quad y\in\Range(\cB)
\end{equation}
and denote $\overline X: = \Pi_{C}(\cB^*\bar y + G)$. Then,
$(\overline X, \bar y)$ solves the KKT system \eqref{eq:KKT-PD}, i.e., $\overline X$ is the unique optimal solution to problem \eqref{eq:newton_proj_P} and $\bar y$ solves problem \eqref{eq:newton_proj_D}. Let $y\in\Range(\cB)$ be any given point. Define the following operator
\[\hat \partial^2\varphi(y):= \cB\partial \Pi_{C}(\cB^*y + G)\cB^*,\]
where $\partial \Pi_{C}(\cB^*y + G)$ is the Clarke subdifferential
 \cite{clarke1990optimization} of the Lipschitz continuous mapping $\Pi_{C}(\cdot)$ at $\cB^*y + G$.
 From \cite{hiriart1984generalized}, we have that
\[\partial^2 \varphi(y)h = \hat \partial^2\varphi(y)h,\quad \forall \,h\in\Re^{2n},\]
where $\partial^2\varphi(y)$ denotes the generalized Hessian of $\varphi$ at $y$, i.e., the Clarke subdifferential of $\nabla \varphi$ at $y$.
Given $X\in\Re^{n\times n}$, define the linear operator $\cU:\Re^{n\times n} \to \Re^{n\times n}$ as follows:
\begin{equation}\label{eq:ssn_U}
\cU(H) : = {\Omega^X}\circ H, \quad \forall \, H\in\Re^{n \times n},
\end{equation}
where $``\circ"$ denotes the Hadamard product of two matrices  and for $1\le i,j\le n$,
\begin{equation}\label{eq:ssncg1_Sigma}
{\Omega_{ij}^X} =  \left\{
\begin{aligned}
&1, \quad {\rm if} \, X_{ij} \ge 0, \\[5pt]
&0, \quad {\rm otherwise.}
\end{aligned}
\right.
\end{equation}
From the definition of the simple polyhedral convex set $C$, it is easy to see that
$\cU\in \partial\Pi_{C}(X)$.

Next, we present an inexact semismooth Newton method for solving problem \eqref{eq:newton_proj_D} and study its global and local convergence.
Since
$\Pi_{C}(\cdot)$ is strongly semismooth as  it is a Lipschitz continuous piecewise affine function
\cite{mifflin1977semismooth,qi1993nonsmooth},
%(whose definition is given next),
we can design a superlinearly or even quadratically convergent semismooth Newton method to solve the
nonsmooth equation \eqref{eq:nablaphi}.
%%%%%%%%%%%%%%%%%
%\begin{definition}[Semismoothness \cite{mifflin1977semismooth,qi1993nonsmooth,sun2002semismooth}]
%  Let $F:\cO \subseteq \cX \rightarrow \cY$ be a locally Lipschitz continuous function on the open set $\cO$ where $\cX$ and $\cY$ are two finite-dimensional real inner product spaces. $F$ is said to be semismooth at $x\in \cO$ if
%  $F$ is directionally differentiable at $x$ and for any $V \in \partial F(x + \Delta x)$ with $\Delta x\rightarrow 0$,
%        \[F(x+\Delta x) - F(x) - V\Delta x = o(\norm{\Delta x}).\] $F$ is said to be strongly semismooth at $x$ if F is semismooth at $x$ and
%        \[F(x+\Delta x) - F(x) - V\Delta x = O(\norm{\Delta x}^2).\]
%    $F$ is said to be a semismooth (respectively, strongly semismooth) function on $\cO$
%if it is semismooth (respectively, strongly semismooth) everywhere in $\cO$.
%\end{definition}
%It is widely known that  continuous piecewise affine functions and twice continuously differentiable functions
%are {all} strongly semismooth everywhere.  In particular, $\Pi_{C}(\cdot)$,
%with $C$ being a polyhedral convex set,  is strongly semismooth
%since it is a Lipschitz continuous piecewise affine function.
%\fbox{\red{I removed the definition of semismoothness.}}

The template of the semismooth Newton conjugate gradient (CG) method for solving \eqref{eq:newton_proj_D} is presented
as follows.

\medskip

\centerline{\fbox{\parbox{\textwidth}{
{\bf Algorithm {\sc Ssncg1}}: {\bf A  semismooth Newton-CG algorithm for solving
\eqref{eq:newton_proj_D}}.
\\[5pt]
Given $\mu \in (0, 1/2)$, $\bar{\eta} \in (0, 1)$, $\tau_1,\tau_2\in(0,1)$, $\tau \in (0,1]$, and $\delta \in (0, 1)$,
choose $y^0\in \Range(\cB)$.
Iterate the following steps for $j=0,1,\ldots:$
\begin{description}
\item[Step 1.]  Choose  $\cU_j\in \partial{\Pi_C}(\cB^*y^j +  {G})$ as given in \eqref{eq:ssn_U}.
 Let $\cV_j := \cB \cU_j \cB^*$ and $\varepsilon_j = \tau_1\min\{\tau_2,\norm{\nabla\varphi(y^j)}\}$.
 Apply the CG algorithm with the zero vector as the starting point to find an approximate solution $d^j$ to the following linear system
\begin{equation}\label{eqn-epsk}
(\cV_j   + \varepsilon_j I_{2n}) d  + \nabla \varphi(y^j) = 0, \quad d\in\Range(\cB)
\end{equation}
such that
\[
\norm{(\cV_j  + \varepsilon_jI_{2n}) d^j + \nabla \varphi(y^j)}\le \min(\bar{\eta}, \| \nabla \varphi(y^j)\|^{{1+\tau}}).
\]
\item[Step 2.] (Line search)  Set $\alpha_j = \delta^{m_j}$, where $m_j$ is the first nonnegative integer $m$ for which
                         \begin{equation*}
                          \varphi(y^j + \delta^{m} d^j) \leq \varphi(y^j) + \mu \delta^{m}
                           \langle \nabla \varphi(y^j), d^j \rangle.
                           \end{equation*}
\item[Step 3.] Set $y^{j+1} = y^j + \alpha_j \, d^j$.
\end{description}
}}}

\medskip

{We note that at each iteration of Algorithm {\sc Ssncg1}, $\cV_j$ is self-adjoint positive semidefinite. Indeed, for $j=0,1,\ldots$, the self-adjointness of $\cV_j$ follows from the self-adjointness of $\cU_j$ and it further holds that
\[\inprod{d}{\cV_j d} = \inprod{d}{\cB \cU_j \cB^* d}
= \sum_{(k,l)\in \Gamma_j} (\cB^* d)_{kl}^2 \ge 0, \quad \forall \, d\in\Re^{2n}, \]
where the last equation follows from the definition of $\cU_j$ given in \eqref{eq:ssn_U} and the index set  $\Gamma_j$ is defined by $\Gamma_j:= \{ {(k,\l)} \mid (\cB^*y^j + G)_{kl}\ge 0,\; 1\le k,l\le n \}$.}
The convergence results for the above {\sc Ssncg1} algorithm are stated in the next theorem.
\begin{theorem}\label{thm:ssncg1_convergence}
Let $\{y^j\}$ be the infinite sequence generated by Algorithm {\sc Ssncg1} for solving problem \eqref{eq:newton_proj_D}. Then, $\{y^j\} \subseteq \Range(\cB) $ is a bounded sequence and any accumulation point $\hat y$ $(\in\Range(\cB))$ of $\{ y^j\} $ is an optimal solution to problem \eqref{eq:newton_proj_D}.
\end{theorem}
\begin{proof}
	Since $\nabla \varphi(y) \in \Range(\cB)$ for any given $y\in\Re^{2n}$, from the properties of the CG algorithm \cite[Theorem 38.1]{trefethen1997numerical},
 we know that for all $j\ge 0$, $d^j \in \Range(\cB)$. Thus, $\{y^j\}\subseteq \Range(\cB)$. All the other results follow directly from
	\cite[Theorem 3.4]{zhao2010newton}.
\end{proof}

Next, we state a theorem on the convergence rate of Algorithm {\sc Ssncg1}. We shall omit the proof here as it can be proved in the same fashion as \cite[Theorem 3.5]{zhao2010newton}.

\begin{theorem}\label{thm:ssncg1_convergencerate}
	%Assume that Assumption \ref{assump:Slater} holds.
	Let $ \bar y$ be an accumulation point of the infinite sequence $\{y^j\}$ generated by Algorithm {\sc Ssncg1} for solving problem \eqref{eq:newton_proj_D}.
	Assume that the following constraint nondegeneracy condition
	\begin{equation}\label{eq:non-degen}
	\cB\,{\rm lin}(\cT_{C}(\widehat G)) = \Range(\cB)
	\end{equation}
	holds at $\widehat G: = \Pi_{C}(\cB^*\bar y + G)$, where
	${\rm lin}(\cT_{C}(\widehat G))$ denotes the lineality space of the tangent cone of $C$ at $\widehat G$. Then, the whole sequence $\{y^j\}$ converges to  $\bar y$ and
	\begin{equation*}
	\| y^{j+1} - \bar y \| = O(\|y^j - \bar y\|^{1+\tau}).
	\end{equation*}
\end{theorem}

\begin{remark}
	\label{rml:finite_term_ssncg1}
	In fact, given the piecewise linear-quadratic structure in problem \eqref{eq:newton_proj_D}, the results given in \cite{fischer1996finite} and
	\cite{sun1998finite} further imply that our Algorithm {\sc Ssncg1} with the Newton linear systems \eqref{eqn-epsk} solved exactly can enjoy a finite termination property.
	Therefore, we can expect to obtain an approximate solution to \eqref{eq:newton_proj_P} through Algorithm {\sc Ssncg1} with the error on the order of the machine precision ({provided the rounding errors introduced by the intermediate
	computations are not amplified significantly}).
\end{remark}
%\subsection{HS-Jacobian with respect to $\mathfrak{B}_n$}

\subsection{Efficient implementations}
%\subsection{Variational properties of $\Pi_{\mathfrak{B}_n}(\cdot)$}
In our implementation of Algorithm {\sc Ssncg1}, the key part is to
solve the  linear system \eqref{eqn-epsk} efficiently. Note that a
similar linear system is also involved in the calculation of the
HS-Jacobian in \eqref{eq:HS-Jocobian_cPH}. Here, we propose to use the conjugate gradient method to solve \eqref{eqn-epsk}. In this subsection, we shall discuss the efficient implementation of the corresponding matrix vector multiplications.

Let
\[V : = B \,{\rm Diag}({\vec}({\Omega}))B^T \;\in \; \Re^{2n\times 2n},
\]
where $B\in\Re^{2n\times n^2}$ denotes the matrix representation of $\cB$ with respect to
the standard basis of $\Re^{n\times n}$ and $\Re^{2n}$ and ${\Omega}$ is given in \eqref{eq:ssncg1_Sigma}. Given $\varepsilon\ge 0$, we shall focus on the following linear system
\begin{equation}\label{eq:vd}
 (V + \varepsilon I_{2n}) d = r.
 \end{equation}
Here, $r\in\Re^{2n}$ is a given vector.
At the first glance, the cost of computing the matrix-vector multiplication $Vd$ for a given vector $d\in\Re^{2n}$
 would be very expensive when the dimension $n$ is large. Fortunately, the matrix $B$ has a  special structure
which we can exploit  to derive a closed form formula for $V$.
Indeed, we have that
\[ B = \left[\begin{array}{c}
e^T \otimes I_n \\
I_n \otimes e^T
\end{array}\right],\]
where ``$\otimes$'' denotes the Kronecker product. Therefore, we can derive the closed form representation of $V$ as follows:
\begin{equation*}
V = \left[\begin{array}{cc}
{\rm Diag}(\Omega e) & {\Omega} \\
({\Omega})^T & {\rm Diag}(({\Omega})^T e)
\end{array}\right].
\end{equation*}
Now it is clear that the computational cost of $Vd$ is only of the order
$\cO(n^2)$. Furthermore, from the $0$-$1$ structure of ${\Omega}$ and a close examination of the sparsity of ${\Omega}$, it is not difficult to show that the computational cost of $Vd$ can further be reduced to $\min\big\{\cO(\gamma + n), \cO(n^2 - \gamma + n)\big\}$, where  $\gamma$ is the number of  nonzero elements in ${\Omega}$.
%\red{Of course, when dealing with a general matrix $B$, a careful analysis of its sparsity pattern is needed to achieve a fast implementation.}
{Following} the terminology used in \cite{lasso}, this sparsity will {be referred to as} the second order sparsity of the underlying projection problem. Similar as is shown in \cite{lasso}, this second order sparsity is the key ingredient for our {efficient} %extremely fast
implementation of Algorithm {\sc Ssncg1}.
Meanwhile, from the above representation of $V$, we can construct a simple preconditioner for the coefficient matrix in \eqref{eq:vd} as follows
\[\widehat V := {\rm Diag}([e^T{(\Omega)}^T  \   e^T{\Omega}]^T) + \varepsilon I_{2n}.\] Clearly, $\widehat V$  will be a good approximation for $V + \varepsilon I_{2n}$ when ${\Omega}$ is a sparse matrix.

%%%%%%%%%%%%%%%%%%%%%%%%%%%%%%%%%%%%%%%
\section{Quadratic programming problems with Birkhoff polytope constraints}
\label{sec:QP_BP}
%Based on the fast computation of the projection onto the Birkhoff polytope and the HS-Jacobian presented in the last section, one can design efficient algorithms for solving various optimization problems with Birkhoff polytope constraints.

As a demonstration on  how one can take advantage  of
the {efficient} %fast
computation of the projection $\Pi_{\mathfrak{B}_n}$ and its HS-Jacobian presented in the last section,
here we show how such {an efficient} %a fast
computation
 can be employed in the design of efficient algorithms for solving the following convex quadratic programming problem:
\begin{equation*}
({\bf P}) \quad \min \left\{f(X): = \frac{1}{2}\inprod{X}{\cQ X} + \inprod{G}{X} + \delta_{\mathfrak{B}_n}(X) \right\},
\end{equation*}
where $\cQ:\Re^{n\times n} \to \Re^{n\times n} $ is a self-adjoint positive semidefinite linear operator, $G\in \Re^{n \times n}$ is a given matrix, $\delta_{\mathfrak{B}_n}$ is the indicator function of $\mathfrak{B}_n$.
Its dual problem in the minimization form is given by
\[({\bf D}) \quad \min \left\{ \delta_{\mathfrak{B}_n}^*(Z) + \frac{1}{2}\inprod{W}{\cQ W}  \mid Z + \cQ W  + G = 0,\, W\in\Range(\cQ)\right\},\]
where $\Range(\cQ)$ denotes the range space of $\cQ$ and $\delta_{\mathfrak{B}_n}^*$ is the conjugate of the indicator function $\delta_{\mathfrak{B}_n}$.
{Similar to the subspace constraint $y\in\Range(\cB)$ in problem \eqref{eq:newton_proj_D}, the constraint $W\in\Range(\cQ)$ ensures the
boundedness of the solution set of ({\bf D}).} Specifically, under this subspace constraint, since $\mathfrak{B}_n$ is a compact set with a nonempty interior, we know that both the primal and dual optimal solution sets are nonempty and compact. In addition, the fact that $\cQ$ is positive definite on $\Range(\cQ)$ further implies that problem ({\bf D}) has a unique optimal solution $(Z^*,W^*)\in \Re^{n\times n}\times\Range(\cQ)$.

%\blue{Similar to the subspace constraint $y\in\Range(\cB)$ in problem \eqref{eq:newton_proj_D}, the constraint $W\in\Range(\cQ)$ ensures the
%boundedness of the solution set of ({\bf D}).}

Equipped with the {efficient} %fast
solver ({\sc Ssncg1} developed in the last section) for computing $\Pi_{\mathfrak{B}_n}(\cdot)$,
it is reasonable for us to use a simple first order method to solve ({\bf P}) and ({\bf D}). For example, one can
adapt the accelerated proximal gradient (APG) \cite{nesterov1983method,beck2009fast} method to solve ({\bf P}) and the classic two block alternating direction method of multipliers   \cite{ADMM1,ADMM2} method with the step-length of 1.618 to solve ({\bf D}).
However, these first order methods may encounter stagnation difficulties or suffer from extremely slow local convergence, especially when one is searching for high accuracy solutions for ({\bf P}) and ({\bf D}).
In order to be competitive against  state-of-the-art interior point method based QP solvers such as those
implemented in Gurobi,
here we propose a semismooth Newton based augmented Lagrangian method for solving
%({\bf P}) and
({\bf D}),
wherein we show how one can take full advantage of the
{efficient} %fast
computation of  $\Pi_{\mathfrak{B}_n}(\cdot)$ and its  HS-Jacobian
to design {a fast} algorithm.

{Here, the main reason for using the dual ALM approach is that	the subproblem in each iteration of the dual ALM is a strongly convex  minimization problem.
	Armed with this critical property, as will be shown in Theorem \ref{thm:ssncg2} and Remark 
	\ref{rmk:ssncg2}, one can naturally apply the inexact semismooth Newton-CG method to solve a reduced problem in the variable $W\in {\rm Ran(\cQ)}\subset \cS^n$
	and the SSNCG method is guaranteed to converge superlinearly (or even quadratically
	if the inexact direction is computed with high accuracy).
    In contrast, if one were to apply the ALM to the
    primal problem, $\min\{ \frac{1}{2}\inprod{X}{\cQ X} + \inprod{G}{X} + \delta_{\mathcal{B}_n}(X)\}$,
    one would first introduce the constraint $X-Y=0$ to make the terms in the objective function separable,
    i.e., $\min\{ \frac{1}{2}\inprod{X}{\cQ X} + \inprod{G}{X} + \delta_{\mathcal{B}_n}(Y) \mid X- Y = 0\}.$
    Then the corresponding reduced subproblem at the $k$th iteration of the primal ALM approach
    would take the following form:
    \begin{eqnarray*}
    	\min \Big\{ \phi_k(X) := \frac{1}{2}\inprod{X}{\cQ X} + \inprod{G}{X}+
    	\frac{\sigma}{2} \norm{ (X+\sigma^{-1} \Lambda^k) - \Pi_{\mathcal{B}_n}(X+\sigma^{-1} \Lambda^k) }^2
    	\mid X \in \cS^n
    	\Big\},
    \end{eqnarray*}
    where $\Lambda^k$ denotes the multiplier corresponding to the constraint $X-Y=0$.
    However, for this reduced subproblem in the variable $X\in \cS^n$,
    the objective function is not necessarily strongly convex when $\cQ$ is singular (especially for the extreme case when $\cQ=0$).
    Therefore, the SSNCG method applied to this
    reduced subproblem in $X$ may not have  superlinear linear convergence.
}

%\blue{Here, we take the dual approach mainly for two reasons. Firstly, extensive numerical experiments on difficult optimization problems
%such as linear semidefinite programming (SDP) \cite{zhao2010newton,SunTY3c,YangST2015} and convex quadratic semidefinite programming (QSDP) \cite{LiSunToh_scb2014,qsdpnal} problems show that the dual approach is more efficient than the primal approach. Secondly, as is already demonstrated in \cite{zhao2010newton,YangST2015, qsdpnal} for SDP and QSDP, the dual ALM approach further allows us to easily incorporate powerful semismooth Newton methods for solving subproblems with superlinear or even quadratic rate of convergence.}
%\red{\fbox{To rewrite}}  \red{Added by Sun Defeng: These two reasons are not valid. The main reason is that, as for LassoNAl, the ALM subproblems have positive definite  Hessians as proven later. Xudong, please re-organize this paragraph}.

 Given $\sigma >0$, the augmented Lagrangian function associated with ({\bf D}) is given as follows:
\[\cL_{\sigma}(Z,W;X) =  \delta_{
	\mathfrak{B}_n}^*(Z) + \frac{1}{2}\inprod{W}{\cQ W}  - \inprod{X}{Z + \cQ W + G} + \frac{\sigma}{2}\norm{Z + \cQ W + G}^2,\]
where $(Z,W,X) \in \Re^{n\times n} \times \Range(\cQ)\times  \Re^{n\times n}$.
The augmented Lagrangian method for solving ({\bf D}) has the following template. {In the algorithm, the notation $\sigma_{k+1} \uparrow \sigma_{\infty} \le \infty$ means that $\sigma_{k+1} \ge \sigma_k$ and the limit of $\{\sigma_k \}$, denoted as $\sigma_{\infty}$, can be some constant finite number or $\infty$. As one can observed later in Theorem \ref{thm:ALM}, the global convergence of Algorithm ALM can be obtained without requiring $\sigma_{\infty} = \infty$.}

\bigskip
\centerline{\fbox{\parbox{\textwidth}{
			{\bf Algorithm ALM}: {\bf An augmented Lagrangian method
				for solving ({\bf D}).}
			\\[5pt]
			Let $\sigma_0 >0$ be a given parameter.
			Choose $(W^0,X^0)\in\Range(\cQ)\times\Re^{n\times n}$
			and $Z^0\in\textup{dom}(\delta^*_{\mathfrak{B}_n})$.
			For $k=0,1,\ldots$, perform the following steps in each iteration:
			\begin{description}
				\item [Step 1.] Compute
				\begin{equation} \label{eq:ALM_WZ}
				(Z^{k+1},W^{k+1})  \approx {\rm argmin} \left\{
				\Psi_k(Z,W):=\mathcal{L}_{\sigma_k}(Z,W; X^{k})
				  \mid\, (Z,W)\in\Re^{n\times n} \times \Range(\cQ)
				\right\}.
				\end{equation}
				\item[Step 2.] Compute
				\begin{equation}\label{eq:ALM_X}
				X^{k+1} = X^k - \sigma_k( Z^{k+1} + \cQ W^{k+1} +  G).
				\end{equation}
				Update $\sigma_{k+1} \uparrow \sigma_\infty\leq \infty$.
			\end{description}
}}}
\bigskip

We shall discuss first the stopping criteria for approximately solving subproblem
\eqref{eq:ALM_WZ}. For any $k\ge 0$, define
\[f_k(X) = -\frac{1}{2}\inprod{X}{\cQ X} - \inprod{X}{G} - \frac{1}{2\sigma_k}\norm{X - X^k}^2,\quad \forall X\in\Re^{n\times n}.\]
Note that $f_k(\cdot)$ is in fact the objective function in the dual problem of \eqref{eq:ALM_WZ}. Let
$\{\varepsilon_k\}$ and $\{\delta_k\}$ be two given positive summable sequences. Given  $X^k\in \Re^{n\times n}$, we propose to terminate solving
 the subproblem \eqref{eq:ALM_WZ} with either one of the following two easy-to-check stopping criteria:
\begin{equation*}
\label{ALM_stop}
\begin{aligned}
&(\textup{A}) \quad
\left\{
\begin{aligned}
&\Psi_k(Z^{k+1},W^{k+1}) - f_k(X^{k+1})
\le \varepsilon_k^2/2\sigma_k, \\[5pt]
& \gamma(X^{k+1})  \le \alpha_k \varepsilon_k/\sqrt{2\sigma_k},
\end{aligned}
\right.
\\[8pt]
&({\rm B})\quad
\left\{
\begin{aligned}
& \Psi_k(Z^{k+1},W^{k+1}) - f_k(X^{k+1}) \le
\delta_k^2 \norm{X^{k+1} - X^k}^2/2\sigma_k, \\[5pt]
& \gamma(X^{k+1})  \le \beta_k\delta_k\norm{X^{k+1} - X^k} /\sqrt{2\sigma_k},
\end{aligned}
\right.
\end{aligned}
\end{equation*}
where $\gamma(X^{k+1}) := \norm{X^{k+1} - \Pi_{\mathfrak{B}_n}(X^{k+1})},$
\[\alpha_k = \min\left\{1,\sqrt{\sigma_k},\frac{\varepsilon_k}{\sqrt{2\sigma_k}\norm{\nabla f_k(X^{k+1})}}\right\}\quad {\rm and} \quad \beta_k = \min\left\{1,\sqrt{\sigma_k},\frac{\delta_k\norm{X^{k+1} - X^k}}{\sqrt{2\sigma_k}\norm{\nabla f_k(X^{k+1})}}\right\}.
\]
From \cite[Proposition 4.3]{cui2016on} and \cite[Lemma 2.2]{qsdpnal},
%\blue{we have  (A) implies {that}
%\[
%%            (\textup{A}')\quad
%\Psi_k(Z^{k+1},W^{k+1}) - \inf \Psi_k
%\le \nu \varepsilon_k^2/2\sigma_k
%\]
%{}and (B) implies that
%\[
%%({\rm B}') \quad
%\Psi_k(Z^{k+1},W^{k+1}) - \inf \Psi_k\leq (\nu
%\delta_k^2/2\sigma_k)\norm{X^{k+1} - X^k}^2,\]
%respectively, where
%$
%\nu = \frac{5}{2} + \frac{1}{2}\lambda_{\max}(\cQ).
%$
%Therefore, based on \cite{rockafellar1976augmented},}
criteria $(A)$ and $(B)$ can be used in ALM to guarantee the global and local convergence of ALM.
{Indeed, from J. Sun's thesis \cite{sun1986thesis} on the investigation of the subdifferentials
of convex piecewise linear-quadratic functions, we know that $\partial f$ is a polyhedral multifunction (see also \cite[Proposition 12.30]{Rockafellar1998variational}). The classic result of Robinson \cite{robinson1981some} on polyhedral multifunctions further implies that Luque's error bound condition \cite[(2.1)]{luque1984asymptotic} associated with $
\partial f$ is satisfied, {i.e., there exist positive constants $\delta$
and $\kappa$ such that
\begin{equation}
\label{eq:errorf}
{\rm dist}(z, \partial f^{-1}(0)) \le \kappa\norm{u}, \quad \forall \, z\in\ \partial f^{-1}(u), \quad \forall \, \norm{u}\le \delta .
\end{equation} }
Thus we can prove  the global and local (super)linear convergence of Algorithm ALM
by adapting the proofs in
\cite[Theorem 4]{rockafellar1976augmented},
\cite[Theorem 2.1]{luque1984asymptotic} and \cite[Theorem 4.2]{cui2016on}.} The next theorem shows that for the convex QP problem ({\bf P}), one can always expect the KKT residual of the sequence generated by the ALM {to converge} at least R-(super)linearly.

Let the objective function $g:\Re^{n\times n}\times \Range(\cQ) \to (-\infty,+\infty]$ associated with ({\bf D}) be
given by
\[g(Z,W) := \delta_{\mathfrak{B}_n}^*(Z) + \frac{1}{2}\inprod{W}{\cQ W}, \quad \forall \, (Z,W)\in \Re^{n\times n}\times \Range(\cQ).\]

\begin{theorem}
	\label{thm:ALM}
	The sequence $\{(Z^k,W^k,X^k)\}$ generated by Algorithm ALM under the stopping criterion $(\textup{A})$
	for all $k\ge0$ is bounded, and $\{X^k\}$ converges to an optimal solution  $X^{\infty}$ of {\rm ({\bf P})}.
	In addition, $\{(Z^k,W^k)\}$ converges to the unique optimal solution {of} {\rm ({\bf D})}. Moreover, for all $k\ge 0$, it holds that
	\begin{equation*}
	\begin{aligned}
	&g(Z^{k+1},W^{k+1}) - \inf \,({\bf D})\\[5pt]
	\le{}
	&\Psi_k(Z^{k+1},W^{k+1}) - \inf \Psi_k + (1/2\sigma_k)(\norm{X^k}^2 - \norm{X^{k+1}}^2).
	\end{aligned}
	\end{equation*}
	
	Let $\Omega$ be the nonempty compact optimal solution set of {\rm ({\bf P})}.
	{Suppose that
		the algorithm is} executed under  {criteria {\rm (A)} and {\rm (B)}} for all $k\ge 0$. {Then, for all $k$ sufficiently large, it holds that
	\begin{eqnarray*}
	%  \begin{aligned}
	{\rm dist}(X^{k+1},  {\Omega}) \le \theta_k {\rm dist}(X^{k}, {\Omega}), \label{eq:asyP} \\[5pt]
%	{\rm dist}((W^{k+1},Z^{k+1}), \Omega_{\rm D}) \le \theta_{\infty}' {\rm dist}(X^{k}, {\Omega}_{\rm D}), \label{eq:asyD} \\[5pt]
	\norm{Z^{k+1}-\cQ W^{k+1} - G} \le \tau_{k}{\rm dist}(X^k,  {\Omega}), \label{eq:asyD_Rsupl}\\[5pt]
	g(Z^{k+1},W^{k+1}) - \inf \,({\bf D}) \le \tau_{k}'{\rm dist}(X^k,  {\Omega}), \label{eq:asygap_Rsupl}
	%  \end{aligned}
	\end{eqnarray*}
where $0\le \theta_k, \tau_k, \tau'_k <1$ and
 $\theta_k \to \theta_{\infty} = \kappa/\sqrt{\kappa^2 + \sigma_{\infty}^2}$,
$\tau_k\to \tau_{\infty} = 1/\sigma_{\infty}$ and $\tau_k'\to \tau'_{\infty} = \norm{X^{\infty}}/\sigma_{\infty}$
with $\kappa$ given in \eqref{eq:errorf}.
%and
%\begin{equation*}
%\begin{aligned}
%%& 1 > \nu\delta_k \to 0,\\[5pt]
%& \theta_k = \big(\kappa/\sqrt{\kappa^2 + \sigma_k^2 }+ 2\nu\delta_k\big)(1 - \nu\delta_k)^{-1} \to \theta_{\infty} = \kappa/\sqrt{\kappa^2 + \sigma_{\infty}^2}
%%\quad (\theta_{\infty} = 0 \; {\rm if}\; \sigma_{\infty} = \infty)
%,\\[5pt]
%& \tau_k = \sigma_k^{-1}(1-\nu\delta_k)^{-1} \to \tau_{\infty} = 1/\sigma_{\infty}
%%\quad (\tau_{\infty} = 0 \; {\rm if}\; \sigma_{\infty} = \infty)
%, \\[5pt]
%& \tau_k' = \tau_k(\nu^2\delta_k^2\norm{X^{k+1} - X^k} + \norm{X^{k+1}} + \norm{X^k})/2 \to \tau'_{\infty} = \norm{X^{\infty}}/\sigma_{\infty}
%%\quad (\tau'_{\infty} = 0 \; {\rm if}\; \sigma_{\infty} = \infty)
%.
%\end{aligned}
%\end{equation*}
Moreover, $\theta_{\infty} = \tau_{\infty} = \tau_{\infty}' = 0$ if $\sigma_{\infty} = \infty$.}
%	where $0\le \theta_{\infty}, \tau_{\infty}, \tau_{\infty}' < 1$ with the property that $ \max\{\theta_{\infty}, \tau_{\infty}, \tau_{\infty}'\} \ll 1 $ if $\sigma_k \to \sigma_{\infty}$ for a sufficiently large $\sigma_\infty$.
\end{theorem}

\begin{remark}
{
%	\red{Besides} the local (super)linear convergence derived in the above theorem, based on the recent advances in \cite[Lemma 3]{cui2017on} and \cite[Theorem 4.1 and Remark 4.1]{zhang2017an}, we can also obtain global (super)linear convergence rate of the dual infeasibility and the duality gaps for the sequence generated by the ALM.
	We also note that if the ALM is used to solve an equivalent reformulation of the primal form of ({\bf P}) and the corresponding subproblems  are solved exactly, {then} the global linear convergence of {a} certain constraint norm to zero can be established via using the results developed in \cite{Chiche2016how,Delbos2005global}.}
\end{remark}

Next, we shall discuss how to solve the subproblems \eqref{eq:ALM_WZ} efficiently. Given $\sigma >0$ and $\widehat X \in \Re^{n\times n}$,
since $\cL_{\sigma}(Z,W;\widehat X)$ is strongly convex on
{$\Re^{n\times n} \times \Range(\cQ)$}, we have that, for any
$\alpha \in \Re$, the level set $\cL_{\alpha}:=\{(Z,W)
\in {\Re^{n\times n} \times \Range(\cQ)} \,\mid\,
\cL_{\sigma}(Z,W;\widehat X) \le \alpha\}$ is a closed and bounded convex set.
Moreover, the optimization problem
\begin{equation}
\label{eq:optALM}
\min \left\{ \cL_{\sigma}(Z,W;\widehat X)\mid (Z,W)
\in {\Re^{n\times n} \times \Range(\cQ)} \right\}
\end{equation}
admits a unique {optimal} solution, which we denote as $(\overline Z,\overline W)
\in {\Re^{n\times n} \times \Range(\cQ)}$.
Define
\[\psi(W) := \inf_Z \cL_{\sigma}(Z, W; \widehat X) \quad {\rm and} \quad Z(W):=
\widehat X - \sigma(\cQ W + G), \quad \forall\, W\in\Range(\cQ).\]
{It is not difficult to see that
	$\inf_{W\in\Range(\cQ)} \varphi(W) = \inf_{Z,{W\in\Range(\cQ)}} \cL_{\sigma}(Z,W;\widehat X)$
	and
\[ \sigma^{-1}\big( Z(W) - \Pi_{\mathfrak{B}_n}(Z(W))\big)
= \arg\min_Z \cL_{\sigma}(Z,W;
\widehat{X}), \quad \forall\, W\in\Range(\cQ). \]
Therefore, $(\overline Z,\overline W)$ solves {the minimization} problem  \eqref{eq:optALM} if and only if %can be computed in the following manner
\begin{eqnarray}
&&\overline W = \arg\min \left\{ \psi(W)\,\mid W\in\Range(\cQ)\right\}, \label{eq:ssncg_W}\\[5pt]
&&\overline Z = \sigma^{-1} \big(Z(\overline W) - \Pi_{\mathfrak{B}_n}(Z(\overline W)) \big) = \arg\min_Z \cL_{\sigma}(Z,\overline W; \widehat X). \nonumber
\end{eqnarray}
Simple calculations show that for all $W\in\Range(\cQ)$,
\[\psi(W) = \frac{1}{2}\inprod{W}{\cQ W} + \frac{1}{\sigma}\inprod{Z(W)}{\Pi_{\mathfrak{B}_n}(Z(W))} - \frac{1}{2\sigma}(\norm{\Pi_{\mathfrak{B}_n}(Z(W))}^2 + \norm{\widehat X}^2).
\]
Note that $\psi$ is strongly convex and continuously differentiable on $\Range(\cQ)$ with
\[\nabla \psi(W) = \cQ W - \cQ \Pi_{\mathfrak{B}_{n}}(Z(W)).\]
Thus, $\overline W$, the optimal solution of \eqref{eq:ssncg_W}, can be obtained through solving the following nonsmooth
piecewise affine equation:
\[\nabla \psi(W) = 0, \quad W\in\Range(\cQ).\]
Given $\widehat W$, define the following linear operator $\cM:\Re^{n\times n} \to \Re^{n\times n}$ by
\begin{equation*}\label{eq:cMssncg2}
\cM(\Delta W):= (\cQ + \sigma \cQ \cP \cQ)\Delta W, \quad \forall \Delta W\in\Re^{n\times n},
\end{equation*}
where $\cP$ is the HS-Jacobian of $\Pi_{\mathfrak{B}_n}$ at $Z(\widehat W)$
{as} given in \eqref{eq:HS-Jocobian_cPH} and {it} is self-adjoint and positive semidefinite. Moreover, since $\cQ$ is self-adjoint and positive definite on ${\Range{\cQ}}$, it follows that $\cM$ is also self-adjoint and positive definite on ${\Range{\cQ}}$.
Similarly as in Section \ref{sec:projBP}, we propose to solve the  subproblem \eqref{eq:ssncg_W} by an inexact semismooth Newton method and $\cM$ will be regarded as a computable generalized Hessian of $\varphi$ at $\widehat W$.}

\medskip

\centerline{\fbox{\parbox{\textwidth}{
			{\bf Algorithm {\sc Ssncg2}}: {\bf A  semismooth Newton-CG algorithm for solving
			\eqref{eq:ssncg_W}}.
			\\[5pt]
			Given $\mu \in (0, 1/2)$, $\bar{\eta} \in (0, 1)$, $\tau \in (0,1]$, and $\delta \in (0, 1)$, choose $W^0 \in\Range(\cQ)$.
			Iterate the following steps for $j=0,1,\ldots:$
			\begin{description}
				\item[Step 1.]
				Let $\cM_j := \cQ + \sigma \cQ \cP_j \cQ$ where $\cP_j$
				is the HS-Jacobian of $\Pi_{\mathfrak{B}_n}$ at
				$Z(W^j)$ given in \eqref{eq:HS-Jocobian_cPH}.
				Apply the CG algorithm to find an approximate solution $dW^j $ to the following linear system
				\begin{equation}\label{eq:ssncg2_dW}
				\cM_j dW + \nabla \psi(W^j) = 0,\quad dW\in \Range(\cQ)
				\end{equation}
				such that
				\[
				\norm{\cM_j dW^j  + \nabla \psi (W^j)}\le \min(\bar{\eta}, \| \nabla \psi(W^j)\|^{{1+\tau}}).
				\]
				\item[Step 2.] (Line search)  Set $\alpha_j = \delta^{m_j}$, where $m_j$ is the first nonnegative integer $m$ for which
				\begin{equation*}
				\psi(W^j + \delta^{m} dW^j) \leq \psi(W^j) + \mu \delta^{m}
				\langle \nabla \psi(W^j), dW^j \rangle.
				\end{equation*}
				\item[Step 3.] Set $W^{j+1} = W^j + \alpha_j \, dW^j$.
			\end{description}
}}}

\medskip

Similar to Theorem \ref{thm:ssncg1_convergence} and Theorem \ref{thm:ssncg1_convergencerate}, it is not difficult to obtain the following theorem on the global and local superlinear (quadratic) convergence for the above Algorithm {\sc Ssncg2}. Its proof is omitted for brevity.

\begin{theorem}
	\label{thm:ssncg2}
	Let  $\{W^j\}$ be the infinite sequence generated by Algorithm {\sc Ssncg2}. Then $\{W^j\}$ converges to the unique optimal solution $\overline W \in \Range(\cQ)$ to problem \eqref{eq:ssncg_W} and
	\[\norm{W^{j+1} - \overline W} = O(\norm{W^j - \overline W}^{1+\tau}).\]
\end{theorem}

\begin{remark}\label{rmk:ssncg2}
	Note that in the above theorem, since $\cQ$ is positive definite on $\Range(\cQ)$, we know that for {each} $j\ge 0$, $\cM_j$ {is} also positive definite on $\Range(\cQ)$.
	Therefore, we do not need any nondegeneracy condition assumption here as is required in Theorem \ref{thm:ssncg1_convergencerate}.
\end{remark}

\begin{remark}
	\label{rmk:ssncg2-rangeQ}
	The restriction of $dW \in \Range(\cQ)$ appears to introduce severe numerical difficulties when we need to solve \eqref{eq:ssncg2_dW}. Fortunately, we can overcome these difficulties via a careful examination of our algorithm and some numerical techniques. Indeed, at the $j$th iteration of Algorithm {\sc Ssncg2}, instead of dealing with \eqref{eq:ssncg2_dW},
	we propose to solve the following simpler linear system
	\begin{equation}\label{eq:dwj}
	(\cI + \sigma \cP \cQ) dW = \Pi_{\mathfrak{B}_n}(Z(W^j)) - W^j,
	\end{equation}
	where $\cI$ is the identity operator defined on $\Re^{n\times n}$. Then, the approximate solution to \eqref{eq:dwj} can be safely used as a replacement of $dW^j$ in the execution of Algorithm {\sc Ssncg2}. We omit the details here for brevity. Interested readers may refer to Section 4 in \cite{qsdpnal} for a detailed discussion
	on why this procedure is legitimate.
\end{remark}

\begin{remark}
	\label{rmk:xkp1}
	{At the $k$th iteration of Algorithm ALM,
	given $X^k$ and $\sigma_k$, we first obtain $W^{k+1}$ via executing Algorithm {\sc Ssncg2}. Then, we have that
	\[Z^{k+1} = \sigma_k^{-1}
	(X^k - \sigma_k(\cQ W^{k+1} + G) - \Pi_{\cB_n}(X^k - \sigma_k(\cQ W^{k+1} + G))).\]
	Therefore, it is easy to see that the multiplier update step
	\eqref{eq:ALM_X} in Algorithm ALM can be equivalently recast as:
\begin{align*}
X^{k+1} = X^k - \sigma_k( Z^{k+1} + \cQ W^{k+1} +  G)
= \Pi_{\cB_n}(X^k - \sigma_k(\cQ W^{k+1} + G)).
\end{align*}}
%	  is used to handle subproblem \eqref{eq:ALM_WZ} in Algorithm ALM,
\end{remark}

\section{Numerical experiments}
\label{sec:numer}

In this section,
we evaluate  the performance of our algorithms from various aspects. We have implemented all our algorithms in
{\sc Matlab}. Unless otherwise  specifically stated, all our computational results are obtained from a 12-core workstation with Intel Xeon E5-2680 processors at 2.50GHz and 128GB
memory. The experiments are run in {\sc Matlab} 8.6 and Gurobi 6.5.2 \cite{Gurobi2016} (with an academic license)
under the 64-bit Windows Operating System.
It is well known that Gurobi is an extremely powerful solver for solving  generic quadratic programming
problems. It is our view that any credible algorithms designed for solving
a specialized class of QP problems should be benchmarked against Gurobi and be able to demonstrate
its advantage over Gurobi. {But we should note  that as a general QP solver, Gurobi does not necessarily fully exploit the specific structure of the Birkhoff polytope, although it can fully exploit the sparsity of
the constraint matrices and variables.}

%%%%%%%%%%%%%%%%%%%%%%
\subsection{Numerical results for the projection onto the Birkhoff polytope}

First we compare our Algorithm {\sc Ssncg1} with the state-of-the-art solver, Gurobi, for solving large scale instances of the projection problems \eqref{eq:newton_proj_P} and its dual \eqref{eq:newton_proj_D}.
{Note that dual problem \eqref{eq:newton_proj_D} is an unconstrained smooth convex optimization problem.
{For solving such a problem, the accelerated proximal gradient (APG) method of Nesterov  \cite{nesterov1983method} has
become very popular due to its simplicity in implementation and strong iteration complexity.
As it is a very natural method for one to adopt in the first attempt
to solve \eqref{eq:newton_proj_D}, we also implement the APG method for solving \eqref{eq:newton_proj_D}
for comparison purposes.}
%Perhaps the most popular and efficient first-order method for solving \eqref{eq:newton_proj_D} is the Nesterov's APG %method \cite{nesterov1983method}. Hence, for comparison purposes, we also implement an APG method for solving
%\eqref{eq:newton_proj_D}.
}

{Recall from \eqref{eq:newton_proj_D} that $C= \{X\in\Re^{n\times n}\mid X\ge 0\}$ and define {the} function $h:\Re^{n\times n}\to\Re$ by $h(Z) = \frac{1}{2}\norm{\Pi_C(Z)}^2$.
Note that
\[ \nabla h(Z) = \Pi_{C}(Z) \quad \mbox{and} \quad \norm{\nabla h(Y) - \nabla h(Z)}\le \norm{Y - Z}, \quad \forall\, Y,Z\in\Re^{n\times n}. \]
Given $\hat y \in\Re^{2n}$, the Lipschitz continuity of $\nabla h$ implies that for all $y\in\Re^{2n}$,
\begin{equation*}
\label{eq:liph}
\frac{1}{2}\norm{\Pi_C(\cB^* y + G)}^2
\le \frac{1}{2}\norm{\Pi_C(\cB^* \hat y + G)}^2
+ \inprod{\Pi_C(\cB^* \hat y + G)}{\cB^* (y-\hat y)}
+ \frac{1}{2}\norm{\cB^* (y-\hat y)}^2.
\end{equation*}
From the above inequality, we can derive the following simple upper bound for $\varphi$:
\begin{equation}
\label{eq:apgvarphi}
 \varphi(y) \le \hat \varphi(y; \hat y) : = \varphi(\hat y) + \inprod{\nabla \varphi(\hat y)}{ y - \hat y}
+ \frac{1}{2} \norm{\cB^* y - \cB^*\hat y}^2, \quad \forall\, y\in\Re^{2n}.
\end{equation}
{The APG method we implemented} here is based on \eqref{eq:apgvarphi}.}
%The APG method used here is based on the simple observation that for any given $\hat y \in\Re^{2n}$, we have
%\[ \varphi(y) \le \hat \varphi(y; \hat y) : = \varphi(\hat y) + \inprod{\nabla \varphi(\hat y)}{ y - \hat y}
%+ \frac{1}{2} \norm{\cB^* y - \cB^*\hat y}^2, \quad \forall y\in\Re^{2n}.\]
The detailed steps of the APG method for solving \eqref{eq:newton_proj_D} are given as follows.

\medskip
\centerline{\fbox{\parbox{\textwidth}{
			{\bf Algorithm APG: {\bf An  accelerated proximal gradient algorithm for \eqref{eq:newton_proj_D}.}}
			\\[5pt]
			Given  $y^0\in \Range(\cB)$,
			set  $z^1 = y^0$ and
			$t_1=1 $.
			For $j = 1,\ldots,$ perform the following steps in each iteration:
			\begin{description}				
				\item[Step 1.] Compute $\nabla \varphi(z^j) = \cB \Pi_{C}(\cB^* z^j + G) - b$. Then
				compute
			      $$
				y^j = {\rm argmin}\left\{ \hat \varphi(y; z^j) \, \mid y\in \Range(\cB) \right\}
				$$
		       via solving the following linear system:
		       \begin{equation}\label{eq:apg_linsys}
		       \cB\cB^* y = \cB\cB^* z^j - \nabla\varphi(z^j),\quad y\in\Range(\cB).
		       \end{equation}
				%%%%%%%%%%%%%%%		
				\item [Step 2.] Set $t_{j+1} = \frac{1+\sqrt{1+4t_j^2}}{2}$, $\beta_j=\frac{t_j-1}{t_{j+1}}$.
				Compute
				$	z^{j+1} = y^j + \beta_j (y^j - y^{j-1}).$
			\end{description}
}}}

\medskip
Note that since $b\in\Range(\cB)$, the solution $y^j\in \Range(\cB)$ to equation \eqref{eq:apg_linsys} is in fact unique. Hence,  Algorithm APG is well defined. In our implementation, we further use the restarting technique to accelerate the convergence of
the algorithm.

 In our numerical experiments, we measure the accuracy of an approximate optimal solution $(X, y)$ for problem \eqref{eq:newton_proj_P} and its dual problem \eqref{eq:newton_proj_D} by using the following relative KKT residual:
\[\eta = \max\{ \eta_P, \eta_{C}\}, \]
where
\[ \eta_P = \frac{\norm{\cB X - b}}{1 + \norm{b}}, \quad
\eta_{C} = \frac{\norm{X - \Pi_{C}(\cB^*y +G)}}{1+\norm{X}}. \]
We note that for the Gurobi solver, the primal infeasibility $\eta_P$ {associated with} the computed approximate solution  is
 usually
very small.
On the other hand,   for  Algorithm {\sc Ssncg1} and Algorithm APG,  since the solution $X$ is obtained through the dual approach, i.e., $X = \Pi_{C}(\cB^*y + G)$, we have that
for these two algorithms, $\eta_C = 0.$

Let $\varepsilon >0$ be a given tolerance. We terminate both algorithms {\sc Ssncg1} and APG
when
$\eta < \varepsilon$.
The algorithms will also be stopped when they reach the maximum number
of iterations (1,000 iterations for {\sc Ssncg1} and 20,000 iterations for
APG) or the maximum computation time of 3 hours.
For the Gurobi solver, we use the default parameter settings, i.e., using the default stopping tolerance and all 12
computing cores.

In this subsection, we test  17 instances of the given matrix $G$
for \eqref{eq:newton_proj_P}
with dimensions $n$ ranging from $10^3$ to $3.2\times 10^4$. Among these test instances,
6 of them are similarity matrices derived from the LIBSVM datasets \cite{chang2011library}: {\bf gisette}, {\bf mushrooms}, {\bf a6a}, {\bf a7a},
{\bf rcv1} and {\bf a8a}.
Similarly as in \cite{wang2010learning}, we first normalize each data point
to have a unit $l_2$-norm and use the following Gaussian kernel to generate $G$, i.e.,
\[ G_{ij} = \exp\left( -\norm{x_i - x_j}^2
\right),\quad \forall\, 1 \le i,j \le n.\]
The other 11 instances are randomly generated using the {\sc Matlab} command:
{\tt G = randn(n)}.

\begin{footnotesize}
	\begin{longtable}{| c c| r |  r | r|}
		\caption{The performance of {\sc Ssncg1}, APG, and Gurobi on the  projection problem \eqref{eq:newton_proj_P} and its dual  \eqref{eq:newton_proj_D}. In the table, ``a'' and ``c1'' stand for APG and {\sc Ssncg1} with the tolerance $\varepsilon = 10^{-9}$; ``b'' stands for Gurobi; ``c2'' stands for {\sc Ssncg1} with $\varepsilon = 10^{-15}$. The entry ``*'' indicates out of memory. The computation time is in the format of ``hours:minutes:seconds''.}\label{table:projBP}
		\\
		\hline
		\mc{2}{|c|}{} &\mc{1}{c|}{} &\mc{1}{c|}{}&\mc{1}{c|}{}\\[-5pt]
		\mc{2}{|c|}{} & \mc{1}{c|}{iter}  &\mc{1}{c|}{$\eta$} &\mc{1}{c|}{time}\\[2pt] \hline
		\mc{1}{|c}{problem} &\mc{1}{c|}{$n$}&\mc{1}{c|}{a $|$ b $|$ c1 $|$ c2}&\mc{1}{c|}{a $|$ b($\eta_P$) $|$ c1 $|$ c2 }
		&\mc{1}{c|}{a $|$ b $|$ c1 $|$ c2}\\ \hline
		\endhead
	rand1
	&  1000	 &1350  $|$  15  $|$  12  $|$  13 	 &   9.7-10 $|$    {\bf 4.1-5} (1.2-15)  $|$    8.7-12 $|$    5.2-16	 &18 $|$ 06 $|$ 01 $|$ 01\\[2pt]
	\hline
	rand2
	&  2000	 &2630  $|$  17  $|$  13  $|$  14 	 &   9.9-10 $|$    {\bf 2.4-5} (1.5-15)  $|$    4.4-12 $|$    4.5-16	 &2:21 $|$ 31 $|$ 02 $|$ 02\\[2pt]
	\hline
	rand3
	&  4000	 &3544  $|$  21  $|$  14  $|$  15 	 &   9.9-10 $|$    {\bf 8.4-6} (2.5-15)  $|$    2.2-13 $|$    4.1-16	 &12:31 $|$ 2:30 $|$ 07 $|$ 08\\[2pt]
	\hline
	rand4
	&  8000	 &6454  $|$  25  $|$  14  $|$  16 	 &   9.9-10 $|$    {\bf 2.3-6}  (1.8-14)  $|$    4.2-10 $|$    4.3-16	 &  1:30:05 $|$ 13:02 $|$ 27 $|$ 34\\[2pt]
	\hline
	rand5
	&  10000	 &8234  $|$  25  $|$  14  $|$  16 	 &   {\bf 3.1-7} $|$    {\bf 6.2-6} (3.4-15)  $|$    8.6-11 $|$    4.5-16	 &  3:00:00 $|$ 21:27 $|$ 44 $|$ 58\\[2pt]
	\hline
	rand6
	&  12000	 &5565  $|$  25  $|$  15  $|$  17 	 &   {\bf 2.6-4} $|$    {\bf 4.6-6} (3.8-15)  $|$    2.0-11 $|$    4.6-16	 &  3:00:00 $|$ 33:31 $|$ 1:14 $|$ 1:33\\[2pt]
	\hline
	rand7
	&  16000	 &3061  $|$  *  $|$  15  $|$  16 	 &   {\bf 1.6-3} $|$  \hspace{0.95cm}*\hspace{1cm}
	$|$    3.7-12 $|$    5.9-16	 &  3:00:02 $|$  \quad*\quad
	  $|$ 2:26 $|$ 2:55\\[2pt]
	\hline
	rand8
	&  20000	 &1646  $|$  *  $|$  16  $|$  17 	 &   {\bf 6.8-3} $|$  \hspace{0.95cm}*\hspace{1cm}
	 $|$    1.7-11 $|$    9.5-16	 &  3:00:07 $|$  \quad*\quad
	   $|$ 4:08 $|$ 4:45\\[2pt]
	\hline
	rand9
	&  24000	 &1014  $|$  *  $|$  16  $|$  17 	 &   {\bf 1.9-2} $|$  \hspace{0.95cm}*\hspace{1cm}
	$|$    2.9-13 $|$    4.9-16	 &  3:00:04 $|$  \quad*\quad
	  $|$ 6:14 $|$ 7:15\\[2pt]
	\hline
	rand10
	&  30000	 &622  $|$  *  $|$  16  $|$  17 	 &   {\bf 4.9-2} $|$  \hspace{0.95cm}*\hspace{1cm}
	$|$    3.8-12 $|$    5.9-16	 &  3:00:14 $|$  \;\;*\;\;
	  $|$ 9:53 $|$ 12:01\\[2pt]
	\hline
	rand11
	&  32000	 &559  $|$  *  $|$  16  $|$  18 	 &   {\bf 6.3-2} $|$  \hspace{0.95cm}*\hspace{1cm}
	$|$    2.0-11 $|$    4.8-16	 &  3:00:17 $|$  \;*\;
	  $|$ 11:57 $|$ 14:10\\[2pt]
	\hline
	gisette
	&  6000	 &928  $|$  24  $|$  11  $|$  12 	 &   9.8-10 $|$    {\bf 3.3-6} (2.5-15)  $|$    9.1-12 $|$    6.5-16	 &7:19 $|$ 6:58 $|$ 14 $|$ 16\\[2pt]
	\hline
	mushrooms
	&  8124	 &763  $|$  20  $|$  11  $|$  13 	 &   9.8-10 $|$    {\bf 9.5-5} (4.8-15)  $|$    2.8-10 $|$    1.9-16	 &11:07 $|$ 11:58 $|$ 27 $|$ 32\\[2pt]
	\hline
	a6a
	&  11220	 &1227  $|$  26  $|$  13  $|$  14 	 &   9.9-10 $|$    {\bf 4.7-6} (4.0-15)  $|$    5.4-12 $|$    3.8-16	 &34:29 $|$ 31:21 $|$ 59 $|$ 1:03\\[2pt]
	\hline
	a7a
	&  16100	 &1377  $|$  *  $|$  14  $|$  15 	 &   9.9-10 $|$  \hspace{0.95cm}*\hspace{1cm}
	 $|$    7.5-13 $|$    2.9-16	 &  1:28:14 $|$  \quad*\quad
	   $|$ 2:14 $|$ 2:34\\[2pt]
	\hline
	rcv1
	&  20242	 &1583  $|$  *  $|$  17  $|$  18 	 &   {\bf 1.3-6} $|$  \hspace{0.95cm}*\hspace{1cm}
	 $|$    2.0-12 $|$    1.9-16	 &  3:00:03 $|$   \quad*\quad
	 $|$ 4:33 $|$ 5:02\\[2pt]
	\hline
	a8a
	&  22696	 &1330  $|$  *  $|$  14  $|$  16 	 &   {\bf 2.7-4} $|$  \hspace{0.95cm}*\hspace{1cm}
	 $|$    9.7-10 $|$    2.5-16	 &  3:00:03 $|$  \quad*\quad
	   $|$ 5:12 $|$ 6:15\\[2pt]
	\hline

	\end{longtable}
\end{footnotesize}

In Table \ref{table:projBP}, we report the numerical results obtained by
{\sc Ssncg1}, APG and Gurobi in solving various instances of the  projection problem \eqref{eq:newton_proj_P}. Here, we terminate algorithms APG and {\sc Ssncg1} when $\eta < 10^{-9}$.
In order  to further demonstrate the ability of {\sc Ssncg1} in computing highly accurate solutions, we also report the results obtained by {\sc Ssncg1} in solving the instances to the  accuracy of $10^{-15}$.
In the table, the first two columns give the name of problems   and the size of $G$ in \eqref{eq:newton_proj_P}. The number of iterations, the relative KKT residual $\eta$ and computation times (in the format hours:minutes:seconds) are listed in the
last twelve columns. For Gurobi,  we also list the relative primal feasibility $\eta_P$.
As one can observe, although Gurobi can produce a very small $\eta_P$, the corresponding relative KKT residual $\eta$ can only reach the accuracy about $10^{-5}$ to $10^{-6}$.
In other words,  comparing to Gurobi, the solutions produced by {\sc Ssncg1} and APG with the tolerance of $\varepsilon = 10^{-9}$ are already more accurate.

One can also observe from Table \ref{table:projBP} that only our algorithm {\sc Ssncg1} can solve all the
test problems to the required accuracies of $\eta < 10^{-9}$ and $\eta < 10^{-15}$. Indeed, APG can only solve 8 smaller instances out of 17 to the desired accuracy after 3 hours and Gurobi reported out of memory when the size of $G$ is larger than 12,000. Moreover, {\sc Ssncg1} is much faster than APG and Gurobi for all the test instances. For example, for the instance {\bf rand4},  {\sc Ssncg1} is at least 26 times faster than
Gurobi and 180 times faster than APG. In addition, {\sc Ssncg1} can solve
{\bf rand11}, a quadratic programming problem with over 1
 billion variables and nonnegative constraints, to the extremely high accuracy of $5\times 10^{-16}$ in about 14 minutes while APG consumed 3 hours to only produce
a solution with an accuracy of $6\times 10^{-2}$. We also emphasize here that from the accuracy of $10^{-9}$ to the much higher accuracy of $10^{-15}$, {\sc Ssncg1} only needs one or two extra iterations and
consumes insignificant additional time.
The latter observation truly confirmed the power of the quadratic (or at least superlinear)
 convergence property of  Algorithm
{\sc Ssncg1} and the power of exploiting the second order sparsity property of the underlying projection problem
within the algorithm.

{Since the worst-case iteration complexity of APG is only sublinear, it is not surprising that the performance of APG is relatively poor compared to {\sc Ssncg1}. We also note that comparing to small scale problems, APG needs much more iterations to obtain relatively accurate solutions for large scale problems. For example, for the instance {\bf rand6}, APG took 3 hours and 5,565 iterations to only generate a relatively inaccurate solution with an accuracy of $3\times 10^{-4}$. Despite this, for small scale instances (especially the instances {\bf gisette}, {\bf mushrooms} and {\bf a6a}), APG, although much slower than {\sc Ssncg1}, can obtain accurate solutions with computation time comparable {to the powerful commercial solver} Gurobi. Thus, as a first-order method, it is already quite powerful.}

Figure \ref{figure:finite-T} plots the KKT residual $\eta$ against the iteration count of {\sc Ssncg1} for solving the
instance {\bf a8a}. Clearly, our algorithm {\sc Ssncg1} exhibits at least a superlinear convergence behavior when approaching the
optimal solution.
%It in fact also indicates that our algorithm {\sc Ssncg1}
%enjoys a finite termination property.
In Figure \ref{figure:dimvstime}, we compare the computational complexities of
{\sc Ssncg1} and Gurobi when used to solve the 17 projection problems in Table \ref{table:projBP}.
It shows that the time  $t$ (in seconds) taken to solve a problem of dimension $n$ is given by
$t = \exp(-16) n^{2.1}$ for {\sc Ssncg1} and $t = \exp(-14) n^{2.3}$  for Gurobi.
 One can further observe that on the average, for a given $n$ in the range
 from $[\exp(6),\exp(11)]$, our algorithm is at least $7n^{0.2}$ times faster than Gurobi.
%\begin{figure}
%	\centering
%	\subfigure[Small Box with a Long Caption]{
%		\label{fig:subfig:a} %% label for first subfigure
%		\includegraphics[width=1.0in]{pic1.eps}}
%	\hspace{1in}
%	\subfigure[Big Box]{
%		\label{fig:subfig:b} %% label for second subfigure
%		\includegraphics[width=1.5in]{pic.eps}}
%	\caption{Two Subfigures}
%	\label{fig:subfig} %% label for entire figure
%\end{figure}

%The reason that we choose $10^{-9}$ instead of lower accuracy is because we want to make sure that the solutions produced by Algorithm {\sc Ssncg1} and Algorithm APG are of higher quality
%comparing to those obtained by Gurobi. Indeed, under this termination criterion, the objective values obtained by Algorithm {\sc Ssncg1} and Algorithm APG are smaller than the objective values obtained by Gurobi.

\begin{footnotesize}
	\begin{longtable}{| c c| c |  c |}
		\caption{The performance of {\sc Ssncg1} and PPROJ on the  projection problems \eqref{eq:newton_proj_P} and its dual \eqref{eq:newton_proj_D}. In the table, ``{pp}'' stands for  PPROJ; ``{c2}'' stands for {\sc Ssncg1} with $\varepsilon = 10^{-15}$. The computation time is in the format of ``hours:minutes:seconds''.}\label{table:projBPHPC}
		\\
		\hline
		\mc{2}{|c|}{} &\mc{1}{c|}{} &\mc{1}{c|}{}\\[-5pt]
		\mc{2}{|c|}{}  &\mc{1}{c|}{$\eta$} &\mc{1}{c|}{time}\\[2pt] \hline
		\mc{1}{|c}{problem} &\mc{1}{c|}{$n$} &\mc{1}{c|}{ pp $|$ c2 }
		&\mc{1}{c|}{pp $|$ c2}\\ \hline
		\endhead
		rand1
		&  1000	 &   {\red{ 7.5-14}} $|$    5.2-16	 &03 $|$ 00\\[2pt]
		\hline
		rand2
		&  2000	 &   {\bf 2.9-12} $|$    4.4-16	 &07 $|$ 02\\[2pt]
		\hline
		rand3
		&  4000	 &   {\bf 1.5-11} $|$    4.0-16	 &32 $|$ 05\\[2pt]
		\hline
		rand4
		&  8000	 &   {\bf 4.0-13} $|$    4.3-16	 &2:32 $|$ 22\\[2pt]
		\hline
		rand5
		&  10000	 &   {\bf 5.8-12} $|$    4.5-16	 &3:45 $|$ 39\\[2pt]
		\hline
		rand6
		&  12000	 &   {\bf 1.6-12} $|$    4.7-16	 &5:31 $|$ 59\\[2pt]
		\hline
		rand7
		&  16000	 &   {\bf 5.8-13} $|$    5.9-16	 &19:34 $|$ 1:20\\[2pt]
		\hline
		rand8
		&  20000	 &   {\bf 4.3-13} $|$    9.5-16	 &34:01 $|$ 2:10\\[2pt]
		\hline
		gisette
		&  6000	 &   {\red{ 1.4-14}} $|$    6.5-16	 &2:39 $|$ 11\\[2pt]
		\hline
		mushrooms
		&  8124	 &   {\bf 6.9-7} $|$    2.0-16	 &  16:22:46 $|$ 21\\[2pt]
		\hline
		a6a
		&  11220	 &   {\red{ 3.3-14}} $|$    3.8-16	 &14:32 $|$ 39\\[2pt]
		\hline
		a7a
		&  16100	 &   {\red{ 3.7-14}} $|$    2.9-16	 &43:53 $|$ 1:21\\[2pt]
		\hline
		rcv1
		&  20242	 &   {\bf 1.3-13} $|$    1.9-16	 &  2:01:32 $|$ 2:19\\[2pt]
		\hline
		\end{longtable}
\end{footnotesize}

\begin{figure}
	\centering
	\subfigure[Convergence behaviour (problem:a8a).]{
	\label{figure:finite-T}
	\includegraphics[width=0.45\textwidth]{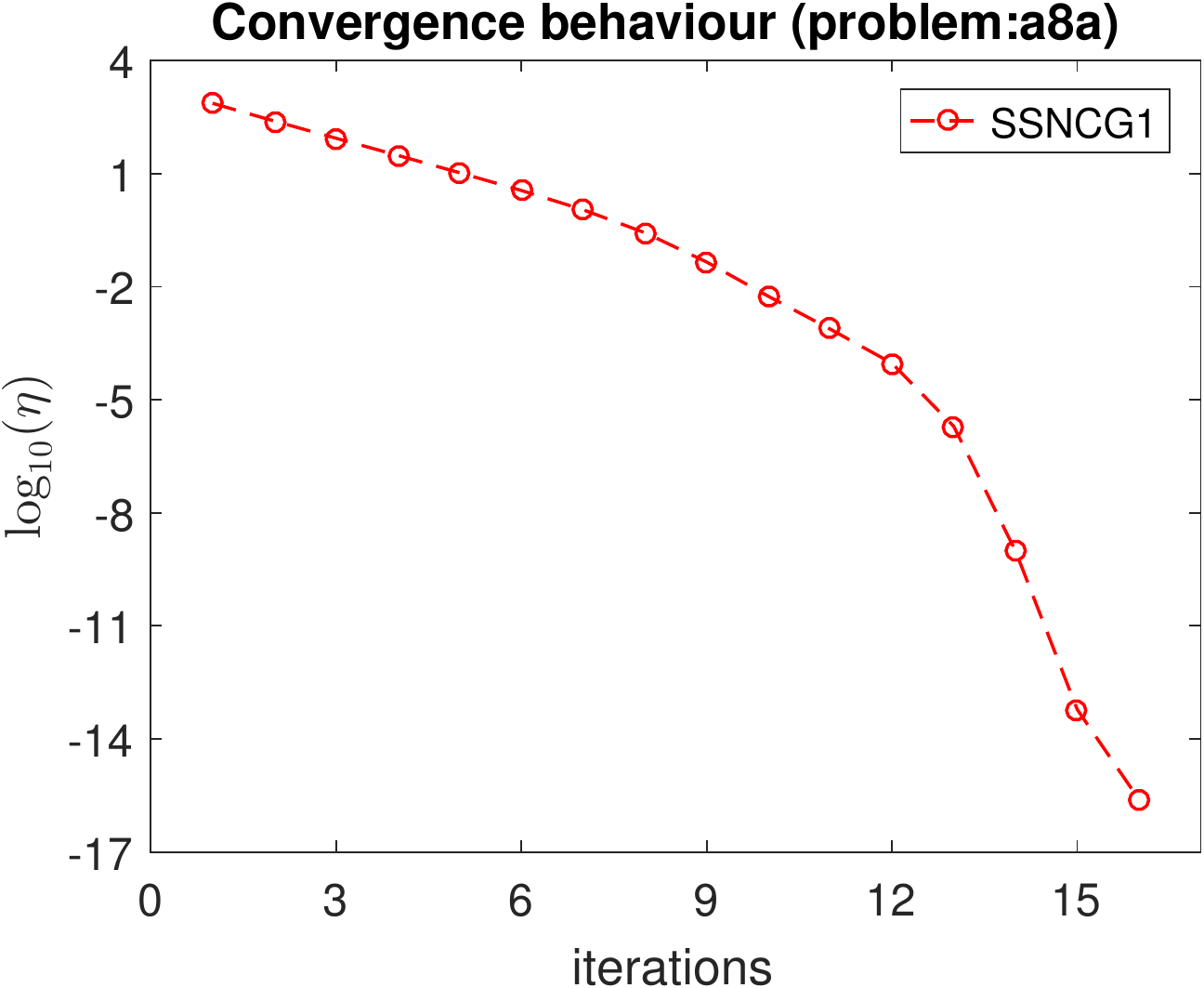}
    }
\hspace{0.5cm}
    \subfigure[{Complexities}: Dimension VS. Time.]{
    	\label{figure:dimvstime}
	\includegraphics[width=0.45\textwidth]{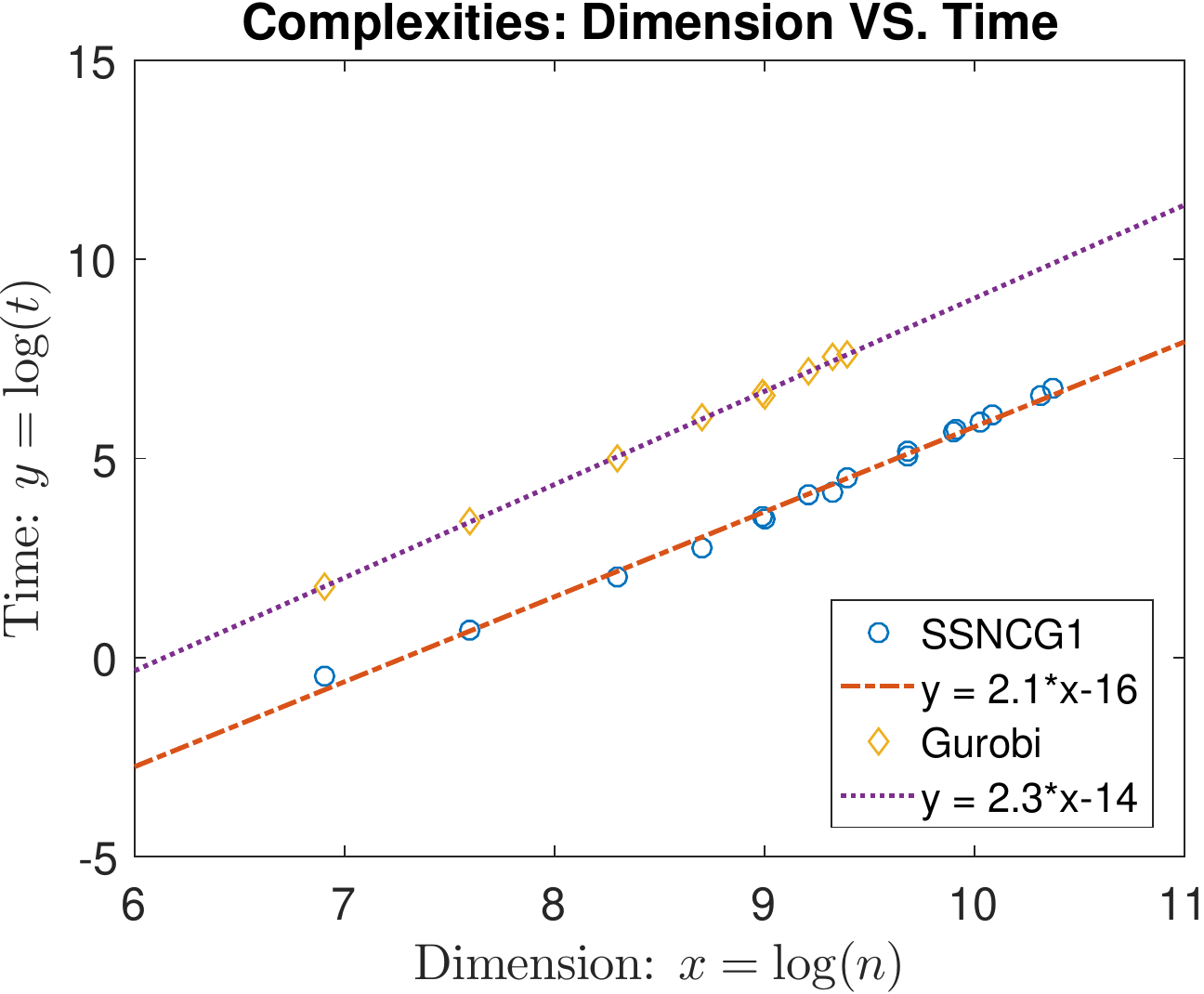}
	}
\caption{Performance {evaluations} of {\sc Ssncg1}.}
\label{figure:1}
\end{figure}

In Table \ref{table:projBPHPC}, we report the detailed results obtained by
our algorithm {\sc Ssncg1} and a recently developed algorithm (called PPROJ) in \cite{zhang2016projection}. PPROJ is an extremely fast implementation of an algorithm which utilizes the sparse reconstruction by separable approximation \cite{wright2009sparse} and the dual active set algorithm (DASA). We used the code downloaded from the authors' homepage\footnote{\url{https://www.math.lsu.edu/~hozhang/Software.html}}. Since PPROJ is implemented in C  and
it depends on some C libraries for linear system solvers, we have to compile PPROJ under the Linux system. Therefore, we compare the performance of {\sc Ssncg1} and PPROJ on the high performance computing (HPC\footnote{\url{https://comcen.nus.edu.sg/services/hpc/about-hpc/}}) cluster at the National University of Singapore. Due to the memory limit imposed for each user, we are only able to test instances with the matrix dimensions less than 21,000.
Default parameter values for PPROJ are used during the experiments. Note that since the stopping criterion of PPROJ is slightly different from ours, we report the accuracy measure $\eta$ corresponding to the solutions obtained by PPROJ.
One can observe that, except the instance {\bf mushrooms}, PPROJ can obtain highly accurate solutions. In fact, we observe from the detailed output file of PPROJ that it does not solve the instance {\bf mushrooms} to the required accuracy while spending excessive amount of time on the DASA in computing  Cholesky factorizations.
One can also observe from Table \ref{table:projBPHPC} that {\sc Ssncg1} is much faster than PPROJ, especially for large scale problems. For example, for the instance {\bf rcv1}, {\sc Ssncg1} is at least 52 times faster than PPROJ. Therefore, one can safely conclude that {\sc Ssncg1} is robust and
{highly} %extremely
efficient for solving projections problems over the Birkhoff polytope. However, we should emphasize here that PPROJ is a solver aiming at computing the projection onto a general polyhedral convex set and does not necessarily fully exploit the specific structure of the Birkhoff polytope.
{In fact, it would be an interesting future research topic to investigate whether
PPROJ can take advantage of both the sparsity of the constraint matrix $B$ and the second order sparsity of the underlying problem to further accelerate its performance.}

\subsection{Numerical results for quadratic programming problems arising from relaxations of QAP problems}

Given matrices $A,B\in \cS^n$, the quadratic assignment problem (QAP) is
given by
$$
\min\{ \inprod{X}{AXB} \mid  X \in\{0,1\}^{n\times n} \cap \mathfrak{B}_n\},
$$
where $\{0,1\}^{n\times n}$ denotes the set of matrices with only $0$ or $1$ entries. It has been shown
in \cite{anstreicher2001new} that a reasonably good lower bound for the above QAP
can often be
obtained by solving the following convex QP  problem:
\begin{eqnarray}\label{eq:qpQAP}
\min \{ \inprod{X}{ \cQ X} \mid X\in \mathfrak{B}_n \},
\end{eqnarray}
where the self-adjoint positive semidefinite linear operator $\cQ$ is defined by
\[ \cQ(X) := AXB - SX - XT,\quad \forall X\in\Re^{n\times n}, \] and $S,T\in\cS^n$ are
given as follows. Consider the eigenvalue decompositions,
$A = V_A D_A V_A^T$, $B=V_B D_B V_B^T$, where
$V_A$ and $D_A={\rm diag}{(\alpha_1,\dots,\alpha_n)}$ correspond to the eigenvectors and eigenvalues of $A$, and $V_B$ and $D_B = {\rm diag}{(\beta_1,\dots,\beta_n)}$
correspond to the eigenvectors and eigenvalues of $B$, respectively.
We assume that $\alpha_1 \geq \ldots \geq \alpha_n$ and
$\beta_1 \leq \ldots \leq \beta_n$.
Let $(\bar{s},\bar{t})$ be an optimal solution to the LP:
$\max \{ e^T s + e^T t\mid s_i + t_j \leq \alpha_i\beta_j,\; i,j=1,\dots,n\}$,
whose solution can be computed analytically as shown in \cite{anstreicher2001new}.
Then $S = V_A{\rm diag}{(\bar{s})}V_A^T$ and $T=V_B {\rm diag}{(\bar{t})} V_B^T$. In our numerical
experiments, the test instances $A$ and $B$ are obtained from the QAP Library \cite{burkard1997qaplib}.
 We measure the accuracy of an approximate optimal solution $X$ for problem \eqref{eq:qpQAP} by using the following relative KKT residual:
\[\eta = \frac{\norm{X - \Pi_{\mathfrak{B}_n}(X - \cQ X)}}{1 + \norm{X} + \norm{\cQ X}}.
\]

Table \ref{table:QAP} reports the performance of the ALM  designed in Section \ref{sec:QP_BP}
against
Gurobi in solving the QP \eqref{eq:qpQAP}.
In the fourth and fifth columns of Table \ref{table:QAP}, ``alm (itersub)'' denotes the
number of outer iterations with itersub in the parenthesis indicating the number of inner iterations
of ALM.
One can observe from Table \ref{table:QAP} that our algorithm is much faster than Gurobi, especially for large scale problems. For example, for the instance {\bf tai150b}, ALM only needs 13 seconds to reach the desired accuracy while Gurobi needs about two and half hours.
One can easily see that Algorithm ALM is highly efficient because
each of its subproblems can be solved by the powerful semismooth Newton-CG algorithm
{\sc Ssncg2} based on the %extremely
efficient computations of $\Pi_{\mathfrak{B}_n}$ and its corresponding HS-Jacobian.
Note that problem \eqref{eq:qpQAP} is in fact a quadratic programming with $n^2$ variables. It is thus not surprising that the interior-point method based solver Gurobi reports out of memory for problem {\bf tai265c} with $n = 256$.

%%%%%%%%%%%%%%%%%%%%%%%%%%%%%%%%%%%
\section{Conclusion}

In this paper, we study the generalized Jacobians in the sense of Han and Sun \cite{han1997newton} of the Euclidean projector over a polyhedral convex set with an emphasis on the Birkhoff polytope. A special element in the set of the generalized Jacobians, referred as the HS-Jacobian,
is successfully constructed. Armed with its simple and explicit formula, we are able to provide a highly efficient procedure to compute the HS-Jacobian. To ensure the efficiency of our procedure, a dual inexact semismooth Newton method is designed and implemented to find the projection over the Birkhoff polytope.
Numerical comparisons between the state-of-the-art solvers Gurobi and PPROJ have convincingly demonstrated the remarkable efficiency and robustness of our algorithm and implementation.
To further demonstrate the importance of the fast computations of the projector and its corresponding HS-Jacobian, we also
incorporate them in the augmented Lagrangian method for solving a class of Birkhoff polytope constrained convex QP problems. Extensive numerical experiments on a collection of QP problems arising from the relaxation of quadratic assignment problems
show the large benefits of our second order nonsmooth analysis based procedure.
%\bibliographystyle{siam}
%\bibliography{DS_proj}

\begin{footnotesize}
	\begin{longtable}{| c c| c |  c | c|}
		\caption{The performance of ALM and Gurobi on the quadratic programming problems \eqref{eq:qpQAP}. In the table, ``{gu}'' stands for Gurobi; ``{alm}'' stands for ALM (accuracy $\eta < 10^{-7}$). The entry ``*'' indicates out of memory. The computation time is in the format of ``hours:minutes:seconds''. ``00'' in the time column means less than 0.5 seconds.}\label{table:QAP}
		\\
		\hline
		\mc{2}{|c|}{} &\mc{1}{c|}{} &\mc{1}{c|}{}&\mc{1}{c|}{}\\[-5pt]
		\mc{2}{|c|}{} & \mc{1}{c|}{iter}  &\mc{1}{c|}{$\eta$} &\mc{1}{c|}{time}\\[2pt] \hline
		\mc{1}{|c}{problem} &\mc{1}{c|}{$n$}&\mc{1}{c|}{gu $|$ alm (itersub)}&\mc{1}{c|}{gu $|$ alm}
		&\mc{1}{c|}{gu $|$ alm}\\ \hline
		\endhead
		lipa50a
		&  50	 & 11  $|$ 21  (58) 	 &   {\red{ 1.8-6}} $|$    7.3-8	 &11 $|$ 01\\[2pt]
		\hline
		lipa50b
		&  50	 & 11  $|$ 17  (123) 	 &   {\red{ 2.2-6}} $|$    5.0-8	 &11 $|$ 05\\[2pt]
		\hline
		lipa60a
		&  60	 & 11  $|$ 19  (54) 	 &   {\red{ 1.4-6}} $|$    6.7-8	 &30 $|$ 01\\[2pt]
		\hline
		lipa60b
		&  60	 & 11  $|$ 18  (104) 	 &   {\red{ 1.7-6}} $|$    9.9-8	 &29 $|$ 05\\[2pt]
		\hline
		lipa70a
		&  70	 & 11  $|$ 19  (52) 	 &   {\red{ 1.7-6}} $|$    6.0-8	 &1:17 $|$ 01\\[2pt]
		\hline
		lipa70b
		&  70	 & 11  $|$ 19  (103) 	 &   {\red{ 1.4-6}} $|$    6.0-8	 &1:20 $|$ 06\\[2pt]
		\hline
		lipa80a
		&  80	 & 11  $|$ 25  (68) 	 &   {\red{ 1.3-6}} $|$    7.3-8	 &2:46 $|$ 01\\[2pt]
		\hline
		lipa80b
		&  80	 & 12  $|$ 18  (141) 	 &   {\blue{ 6.3-7}} $|$    9.3-8	 &2:52 $|$ 14\\[2pt]
		\hline
		lipa90a
		&  90	 & 11  $|$ 20  (54) 	 &   {\red{ 2.7-6}} $|$    8.8-8	 &5:32 $|$ 01\\[2pt]
		\hline
		lipa90b
		&  90	 & 12  $|$ 19  (134) 	 &   {\blue{ 5.5-7}} $|$    2.5-8	 &5:46 $|$ 15\\[2pt]
		\hline
		sko100a
		&  100	 & 14  $|$ 26  (95) 	 &   {\red{ 8.5-6}} $|$    8.5-8	 &2:06 $|$ 11\\[2pt]
		\hline
		sko100b
		&  100	 & 14  $|$ 27  (93) 	 &   {\red{ 8.3-6}} $|$    7.9-8	 &2:06 $|$ 10\\[2pt]
		\hline
		sko100c
		&  100	 & 15  $|$ 27  (93) 	 &   {\red{ 4.5-6}} $|$    9.0-8	 &2:11 $|$ 11\\[2pt]
		\hline
		sko100d
		&  100	 & 15  $|$ 26  (91) 	 &   {\red{ 4.8-6}} $|$    8.8-8	 &2:06 $|$ 10\\[2pt]
		\hline
		sko100e
		&  100	 & 14  $|$ 27  (98) 	 &   {\red{ 5.8-6}} $|$    8.5-8	 &2:06 $|$ 11\\[2pt]
		\hline
		sko100f
		&  100	 & 16  $|$ 27  (93) 	 &   {\red{ 6.1-6}} $|$    9.6-8	 &2:15 $|$ 09\\[2pt]
		\hline
		sko64
		&  64	 & 13  $|$ 27  (91) 	 &   {\red{ 7.3-6}} $|$    9.0-8	 &13 $|$ 04\\[2pt]
		\hline
		sko72
		&  72	 & 13  $|$ 26  (86) 	 &   {\red{ 8.1-6}} $|$    7.6-8	 &22 $|$ 04\\[2pt]
		\hline
		sko81
		&  81	 & 14  $|$ 26  (89) 	 &   {\red{ 4.4-6}} $|$    7.6-8	 &43 $|$ 06\\[2pt]
		\hline
		sko90
		&  90	 & 14  $|$ 26  (95) 	 &   {\red{ 4.4-6}} $|$    7.8-8	 &43 $|$ 08\\[2pt]
		\hline
		tai100a
		&  100	 & 11  $|$ 18  (52) 	 &   {\red{ 1.3-6}} $|$    9.5-8	 &10:31 $|$ 02\\[2pt]
		\hline
		tai100b
		&  100	 & 11  $|$ 27  (98) 	 &   {\red{ 1.3-6}} $|$    9.1-8	 &10:31 $|$ 13\\[2pt]
		\hline
		tai50a
		&  50	 & 11  $|$ 20  (55) 	 &   {\red{ 1.1-6}} $|$    6.1-8	 &09 $|$ 01\\[2pt]
		\hline
		tai50b
		&  50	 & 13  $|$ 25  (89) 	 &   {\red{ 5.9-6}} $|$    8.5-8	 &10 $|$ 03\\[2pt]
		\hline
		tai60a
		&  60	 & 10  $|$ 19  (54) 	 &   {\red{ 5.4-6}} $|$    9.6-8	 &27 $|$ 01\\[2pt]
		\hline
		tai60b
		&  60	 & 10  $|$ 28  (102) 	 &   {\red{ 5.4-6}} $|$    6.6-8	 &27 $|$ 06\\[2pt]
		\hline
		tai80a
		&  80	 & 11  $|$ 21  (59) 	 &   {\red{ 1.2-6}} $|$    7.9-8	 &2:36 $|$ 01\\[2pt]
		\hline
		tai80b
		&  80	 & 11  $|$ 27  (98) 	 &   {\red{ 1.2-6}} $|$    8.5-8	 &2:36 $|$ 07\\[2pt]
		\hline
		tai256c
		&  256	 & *  $|$  2  ( 4) 	 & *  $|$    2.1-16	 & *  $|$ 00\\[2pt]
		\hline
		tai150b
		&  150	 & 19  $|$ 27  (94) 	 &   {\green{ 4.3-7}} $|$    9.3-8	 &  2:46:17 $|$ 13\\[2pt]
		\hline
		tho150
		&  150	 & 16  $|$ 24  (96) 	 &   {\red{ 5.6-6}} $|$    9.9-8	 &18:52 $|$ 22\\[2pt]
		\hline
		wil100
		&  100	 & 13  $|$ 25  (82) 	 &   {\red{ 9.1-6}} $|$    8.8-8	 &2:14 $|$ 07\\[2pt]
		\hline
		wil50
		&  50	 & 13  $|$ 29  (99) 	 &   {\red{ 3.9-6}} $|$    8.4-8	 &05 $|$ 03\\[2pt]
		\hline
		esc128
		&  128	 & 17  $|$  2  ( 4) 	 &   8.2-11 $|$    2.2-16	 &09 $|$ 00\\[2pt]
		\hline
		
	\end{longtable}
\end{footnotesize}

\section*{Acknowledgment}
{We would like to thank Professor Jong-Shi Pang at   University of Southern California for his helpful comments on an early version of this paper and the  referees for helpful suggestions to  improve the quality of this paper.}
%\red{We also thank the referees for helpful suggestions to improve the paper.}

%%%%%%%%%%%%%%%%%%%%%%%%%%%%%%%

\end{document}